\documentclass[reqno]{amsart}
\usepackage{amsaddr}
\usepackage{hyperref}

\title{On a Repulsion-Diffusion Equation with Immigration}
\author{Peter Koepernik\textsuperscript{1,2}}
\thanks{\textsuperscript{2}The author has benefitted while working on this article from an EPSRC grant EP/W523781/1.}
\address{\textsuperscript{1}Department of Statistics, University of Oxford, 24–29 St Giles’, Oxford, OX1 3LB, United Kingdom}
\email{peter.koepernik@stats.ox.ac.uk}
\date{17 February 2023}

\usepackage{amsmath}  
\usepackage{IEEEtrantools}
\usepackage{amssymb}  
\usepackage{mathtools} 
\usepackage{dsfont}
\usepackage{hyperref} 
\usepackage[capitalise]{cleveref}

\usepackage{tikz}

\calclayout

\usepackage{enumitem}
\setenumerate[0]{label=(\roman*)}
\setlist{topsep=8pt,itemsep=4pt,partopsep=4pt, parsep=4pt}

\crefname{equation}{}{}
\Crefname{equation}{Eq.}{Eqs.}

\theoremstyle{plain}
\newtheorem{theorem}{Theorem}[section]
\newtheorem{lemma}[theorem]{Lemma}
\newtheorem{proposition}[theorem]{Proposition}
\newtheorem{corollary}[theorem]{Corollary}

\crefname{lemdef}{Lemma}{Lemmas}
\Crefname{lemdef}{Lemma}{Lemmas}

\theoremstyle{remark}
\newtheorem{remark}[theorem]{Remark}

\theoremstyle{definition}
\newtheorem{assumption}{Assumption}[section]

\crefname{assumption}{}{}
\numberwithin{equation}{section}

\newcommand{\R}{\mathbb{R}} 
\newcommand{\N}{\mathbb{N}} 

\newcommand*\diff{\mathop{}\!\mathrm{d}}
\newcommand*\dd{\diff}

\newcommand{\ind}{\mathds{1}}

\newcommand{\e}{\mathrm{e}}

\newcommand{\nnorm}[1]{{\left\vert\kern-0.25ex\left\vert\kern-0.25ex\left\vert #1
    \right\vert\kern-0.25ex\right\vert\kern-0.25ex\right\vert}}

\begin{document}

\begin{abstract}
    We study a repulsion-diffusion equation with immigration and linear diffusion, whose asymptotic behaviour is related to stability of long-term dynamics in spatial population models and other branching particle systems. We prove well-posedness and find sharp conditions on the repulsion under which a form of the maximum principle and a strong notion of global boundedness of solutions hold. The critical asymptotic strength of the repulsion is $|x|^{1-d}$, that of the Newtonian potential.
\end{abstract}

\maketitle
\section{Introduction}

We consider the following partial differential equation,
\begin{equation}\label{eq:classicalpde}
        \partial_t \rho = \frac{1}{2}\Delta \rho + \nabla\cdot  (\rho \nabla W\star \rho) + f,
\end{equation}
on $\R^d$ for any dimension $d\in \N$, started from a non-negative, bounded and integrable initial condition, where
\begin{enumerate}
    \item $f$ is the \emph{immigration}, which is bounded, integrable, and non-negative,
    \item $W \in C^2(\R^d\setminus \left\{ 0 \right\} )$ is the \emph{interaction potential}, which we assume to be eventually decreasing.
\end{enumerate}
Our main arguments work in the case where $f$ is time-dependent with sufficient regularity, see \cref{rem:timedependentf}.

\subsection{Motivation}

One way to interpret \cref{eq:classicalpde} is as the hydrodynamic limit of a branching particle system (BPS) in which particles appear in $\R^d$ at a constant rate distributed according to $f$, move along the paths of independent Brownian motions, and at constant rate branch into a random number of offspring with mean one. Additionally, a particle at $x$ experiences a drift $-\nabla W(x-y)$ if there is another particle at $y$. Since $W$ is eventually decreasing, this interaction will be repulsive at long range. Then the solution of \cref{eq:classicalpde} describes the population density of this system in the mean-field limit of a large number of individuals. See \cite{meanfield1,meanfield2} for mean-field arguments in a similar setting.

BPS with immigration (with or without repulsion or other interactions) appear in many contexts from biology and physics. They can be used to model air showers of particles produced by extraterrestrial cosmic rays entering the atmosphere \cite{airshower}, families of neutrons in subcritical nuclear reactors, which sustain the reaction with a constant stream of neutrons from an outside source \cite{nuclear}, or biological populations \cite{dawsonsuperprocess,dynkin1991branching,li1992measure}. Immigration in the biological context could arise from different sources. One might consider a steady flow of individuals from a large, stable population migrating into a new, uncontested habitat.
Secondly, a BPS as considered here goes extinct in finite time almost-surely, but conditioned on survival it looks exactly like a BPS with a certain kind of immigration \cite{evansimmortal}. We go into more detail on this particular example in \cref{app:superprocesses}.

In many of these examples, it is of interest whether the system has stable long-term dynamics, which in the mean-field limit is reflected in the asymptotic behaviour of $\rho_t$. If we consider first the case where there is no repulsion, then the equation reduces to $\partial_t \rho = \frac{1}{2}\Delta \rho + f$, and the solution (started from zero) is given by \[
\rho_t = \int_0^t G_s \star f \dd s,
\] where $G_s(x) = (2\pi s)^{-d / 2} \e^{-x^2 / (2s)},\, s>0,$ is the heat kernel. Provided that $f\not\equiv 0$:
\begin{enumerate}
    \item If $d \le 2$, then $\rho_t \uparrow \infty$ locally uniformly.
    \item If $d \ge 3$, then $\rho_t$ converges to a bounded stationary distribution. 
\end{enumerate}
This is proved in \cref{lem:whenWis0}. Indeed it is a known fact that critical branching processes in one and two dimensions tend to be unstable in the sense that, after a long time, they have either gone extinct, or they have a lot of mass that concentrates in large ``clumps'' \cite{felsenstein,kallenbergtree}. We elaborate on this phenomenon in \cref{app:superprocesses}. One of the reasons this does not occur in the real world is that individuals tend to migrate away from overcrowded areas, which can be modelled by a pairwise repulsion between individuals. This motivates the question whether solutions to \cref{eq:classicalpde} remain asymptotically bounded in dimensions $d\le 2$ if the repulsion is sufficiently strong.

\subsection{Related Work}

\subsubsection*{SuperBrownian Motion with Immigration}
A lot of work has been done on superBrownian motion with immigration (and no interaction), which is a measure-valued stochastic process that formally solves \cref{eq:classicalpde} without the interaction and an additional term $\sqrt{\rho_t} \diff W$ for a space-time white noise $W$. It can be obtained from the same BPS whose hydrodynamic limit is given by \cref{eq:classicalpde}, except that particles don't interact and the branching rate is scaled up simultaneously with the particle density. We go into more detail on superBrownian motion in \cref{app:superprocesses}. Amongst the known results for this process are central limit theorems \cite{clt} and large deviation principles for large \cite{ldp1,ldp2} and small times \cite{ldpsmall}. Further results on this and more general measure-valued diffusions with immigration can be found in \cite{immigrationstructures} and references therein.

\subsubsection*{Branching Brownian Motion with Interaction}
Some work has also been done on branching Brownian motion (BBM) with interaction. Classical BBM, without interaction, is a system of particles that move as independent Brownian motions, and branch at constant rates into exactly two offspring. The number of particles grows exponentially and there is no chance of extinction, which makes it quite different from the BPS with critical branching considered here. Questions of interest in this setting are often of extremal type, such as the structure of the process close to the furthest particle from the origin \cite{bbmextremal2,bbmextremal1}, or large time limits of the population's empirical measure if scaled appropriately \cite{englander1,englander2}. Some authors have studied BBM with repulsive or attractive interactions. Engl\"ander considered the case where particles have an Ornstein-Uhlenbeck-type attraction or repulsion (that is, $W(x) = b x^2$ for $b\in \R$) to or from their common centre of mass \cite{englander1}, and from each other \cite{englander2}. Similar results have been obtained in the context of supercritical superBrownian motion \cite{gillOU}. This particular interaction often allows for explicit calculations because $\nabla W(x) \propto x$ is linear. Note also that in light of our motivation, it is not the most natural choice of repulsion, since its strength grows, rather than decays with distance. In another recent paper \cite{bovierbbm}, authors study a BBM in which they introduce short-range pairwise repulsion through a change of measure that penalises the total time that particles spend within close range of each other. They show that the dominant effect of the penalisation is a drastic reduction in branching rate, and that this model is well-approximated by a simplified model in which only branching events are penalised, and there is no repulsion between individuals once they are born.

\subsubsection*{Aggregation-Diffusion Equations}
Equation \cref{eq:classicalpde} is also related to a well-studied class of non-local, nonlinear partial differential equations known as \emph{aggregation-diffusion equations},
\begin{equation}\label{eq:aggregationdiffusion}
    \partial_t \rho = \frac{1}{2}\Delta \rho + \nabla \cdot (\rho \nabla W\star\rho),
\end{equation} which differ from \cref{eq:classicalpde} in that there is no immigration, and the interaction potential is attractive at long range rather than repulsive. We refer to \cite{pdereview,horstmann} for reviews of this class of equations. Variants of \cref{eq:aggregationdiffusion} also exist with non-linear diffusion \cite{nonlinear1,nonlinear2}; when the diffusion is linear as above, \cref{eq:aggregationdiffusion} is commonly called a \emph{McKean Vlasov equation}. Aggregation-diffusion equations have attracted significant interest in the literature because they describe the large scale dynamics of a wide variety of interacting particle systems arising in biology, physics, social and other life sciences, which are often driven by long-range attraction and short range repulsion. Examples include chemotaxis, bacteria orientation, or motion of human crowds, see \cite{chemotaxis,patlak,pdereview}. Significant mathematical interest is further due to the delicate competition between aggregation and diffusion, which leads to a dichotomy between well-posedness and finite time blowup \cite{blanchet,wellposed1,wellposed2,wellposed3,wellposed4,wellposed5,herrero,blowup1,blowup2}. For sufficiently weak interaction, the diffusion dominates and the solution asymptotically simplifies to the solution of the heat equation \cite{joseheatkernel}, while a balance between diffusion and aggregation can lead to the existence of non-trivial steady states \cite{stationary1,stationary2,stationary3,stationary4,stationary5}.

Even though \cref{eq:classicalpde} looks similar to \cref{eq:aggregationdiffusion}, its behaviour is markedly different. Where the competition between aggregation and diffusion decides the behaviour of \cref{eq:aggregationdiffusion}, that of \cref{eq:classicalpde} is decided by the competition between the immigration against the diffusion \emph{and} the repulsion, which both work to spread the immigrated mass.

\subsection{Summary of Results}

We first establish well-posedness of \cref{eq:classicalpde} under mild regularity assumptions on $f$ and $W$. Then we find sharp conditions on $W$ under which the following global boundedness property holds:
\begin{equation}\label{eq:cheapMprop}
    \forall \rho_0\colon \left\|\rho_0\right\|_\infty \le M \implies \sup_{t\ge 0} \left\|\rho_t\right\|_\infty \le M,
\end{equation}
for an explicit value $M > 0$ which is also sharp. Here $\left\|\cdot \right\|_\infty = \left\|\cdot \right\|_{L^\infty}$ denotes the supremum norm. Under the same (sharp) conditions on $W$, and for the same value $M$, a form of the maximum principle holds:
\begin{equation}\label{eq:cheapMP}
     \forall \rho_0\colon \left\|\rho_0\right\|_\infty \ge M \implies \max_{t\ge 0} \left\|\rho_t\right\|_\infty = \left\|\rho_0\right\|_{\infty}.
\end{equation}
In particular, \cref{eq:cheapMprop,eq:cheapMP} give sufficient conditions on $W$ under which $\sup_{t\ge 0} \left\|\rho_t\right\|_\infty < \infty$ for any bounded initial condition.
Both results follow from a differential inequality of the form \[
    \partial_t^+ \left\|\rho_t\right\|_\infty \le \left\|f\right\|_\infty - c \left\|\rho_t\right\|_\infty^2,
\] for a sharp value $c = c_W \in \R$, which implies \cref{eq:cheapMprop,eq:cheapMP} if $c_W > 0$. See \cref{thm:MP,thm:Mproperty,thm:partial+} for precise statements. It will turn out that the critical strength of the repulsion for $c_W > 0$ to hold is $|\nabla W(x)| \sim |x|^{1-d}$, and $W$ must be singular at the origin. A natural example that satisfies both assumptions is the \emph{Newtonian potential},
\begin{equation}\label{eq:newtonian}
    W_N(x) = c_d^{-1} \begin{cases}
        \frac{1}{d-2}|x|^{2-d} ,& d \neq  2,\\
        -\log |x| ,& d = 2,
    \end{cases}
\end{equation}
where $c_d$ is the surface area of the unit ball in $\R^d$. The Newtonian potential is the Green's function of the Laplace equation (that is, $\Delta W_N = -\delta_0$ in a distributional sense), and has the physical interpretation of the electrodynamical repulsive potential in $\R^d$.

\subsection{Outline}
In the following section we will give precise statements of our results, and \cref{sec:proofs} contains the proofs, followed by a short outlook in \cref{sec:outlook}. In \cref{app:superprocesses} we elaborate on the connection between \cref{eq:classicalpde} and stability of long-term dynamics in spatial population models, and in \cref{app:sobolev} we give a brief definition of fractional Sobolev spaces, and recall and proof some basic facts about them.

\section{Results}\label{sec:results}

Our results require the following regularity assumptions. By $\mathcal{W}^{\gamma,p}$ we denote the usual Sobolev spaces on $\R^d$, where $p\in [1,\infty]$, and $\gamma\ge 0$ may not be an integer. If $\gamma = 0$ then $\mathcal{W}^{\gamma,p} = L^p$ is the usual $L^p$ space. See \cref{app:sobolev} for a brief and \cite{hitchhiker,sobolev2,sobolev3} for a comprehensive introduction to fractional Sobolev spaces.  We further write $g_+ = g \vee 0$ and $g_- = (-g) \vee 0$ for the positive and negative part of a function $g\colon \R^d \to \R$, respectively.

\begin{assumption}
    \label{ass}
    \textup{(i)} There is $\gamma_f \in (0,1)$ such that $f \in \mathcal{W}^{\gamma_f,1} \cap \mathcal{W}^{\gamma_f,\infty}$.
    \begin{enumerate}
        \setcounter{enumi}{1}
        \item $\nabla W$ and $\Delta W$ are bounded on $\R^d\setminus B(0,r)$ for all $r > 0$, and locally integrable.
        \item At least one of $(\Delta W)_-$ and $(\Delta W)_+$ is integrable.
    \end{enumerate}
\end{assumption}

\cref{ass}(i) is required to ensure that the mild solution is also a classical solution. A simple sufficient criterion is that $f$ is compactly supported and Hölder continuous (with any positive exponent). \cref{ass}(iii) ensures that $\int \Delta W \coloneqq \int_{\R^d} \Delta W(x) \diff x\in [-\infty,\infty]$ exists. As we will see shortly, the fact that $W$ is eventually decreasing implies $\int \Delta W < \infty$ and hence that in fact $(\Delta W)_+$ must be integrable.
Here and later, the integral is in a strict, classical sense; for example the Newtonian potential has $\Delta W_N \equiv 0$ outside the null set $\left\{ 0 \right\} $, so $\int \Delta W_N = 0$. The following well-posedness result only requires \cref{ass}(i), and \cref{ass}(ii) for $\nabla W$, but no assumptions on $\Delta W$.

\begin{theorem}[Well-posedness]\label{thm:wellposedness}
    For every initial condition there exists a unique mild solution to \cref{eq:classicalpde} up to a maximal existence time $T^\star \in (0,\infty]$. It is non-negative and solves \cref{eq:classicalpde} in the classical sense on $(0,T^\star)$.
\end{theorem}
In this statement and the remainder of \cref{sec:results}, we only consider initial conditions that are non-negative, bounded, and integrable, without mentioning this explicitly. Precise statements are in \cref{thm:wellposed,thm:regularity} below.  We note that it is not possible to prove $T^\star = \infty$ under the current assumptions, because the short-range behaviour of $W$ could lead to finite-time blowup. 

The main result of this work is the identification of an index $c_W \in \R$ that plays a critical role in the behaviour of solutions to \cref{eq:classicalpde}. We first state our main results in terms of $c_W$, and discuss its definition and properties afterwards.

\begin{proposition}\label{prop:globalexistence}
    If $c_W > 0$, then $T^\star = \infty$ for every initial condition. 
\end{proposition}

\begin{theorem}[Maximum Principle]\label{thm:MP}
    If $c_W > 0$ and $M = \sqrt{\left\|f\right\|_\infty / c_W} $, then
    \begin{equation}\label{eq:MP}
    \left\|\rho_0\right\|_\infty \ge M \implies \max_{0\le t < T^\star} \left\|\rho_t\right\|_\infty = \left\|\rho_0\right\|_{\infty}
    \end{equation}
    for every initial condition $\rho_0$. If $M < \sqrt{\left\|f\right\|_\infty / c_W} $, or $c_W < 0$ and $M > 0$ arbitrary, then there exists an initial condition $\rho_0$ for which \cref{eq:MP} is false. 
\end{theorem}
\begin{theorem}[Global Boundedness]\label{thm:Mproperty}
    If $c_W > 0$ and $M = \sqrt{\left\|f\right\|_\infty / c_W} $, then
    \begin{equation}\label{eq:Mproperty}
         \left\|\rho_0\right\|_\infty \le M \implies \sup_{0 \le t < T^\star} \left\|\rho_t\right\|_\infty \le M
    \end{equation}
    for every initial condition $\rho_0$. If $M < \sqrt{\left\|f\right\|_\infty / c_W} $, or $c_W < 0$ and $M > 0$ arbitrary, then there exists an initial condition $\rho_0$ such that \cref{eq:Mproperty} is false. 
\end{theorem}

Both theorems are consequences of the following result. Write $\partial_t^+ g(t) \coloneqq \varlimsup_{h \downarrow 0} \frac{g(t+h) - g(t)}{h}$ for a function $g\colon [0,T) \to [0,\infty)$ and $t \in [0,T)$.

\begin{theorem}\label{thm:partial+}
    For every initial condition and $t \in [0,T^\star)$,
    \begin{equation}\label{eq:partial+}
        \partial_t^+ \left\|\rho_t\right\|_\infty \le \left\|f\right\|_\infty - c_W \left\|\rho_t\right\|_\infty^2.
    \end{equation}
    For every $c > c_W$, there exists an initial condition such that \cref{eq:partial+} is false at time zero if $c_W$ is replaced by $c$.
\end{theorem}

The differential inequality \cref{eq:partial+} can be turned into an explicit upper bound, which implies the positive assertions of \cref{thm:MP,thm:Mproperty}. See \cref{cor:partial+} for a precise statement, and \cref{fig:tanhbound} for an illustration. We remark that counterexamples for $\rho_0$ in the negative statements of all three theorems can be chosen to be infinitely differentiable and compactly supported. 

\begin{figure}
    \centering
    \begin{tikzpicture}[scale=1.7]
        \def\tmax{4}
        \def\ymax{2}
        \def\M{1}
        \draw[->] (0,0) -- (\tmax,0) node[anchor=west] {$t$};
        \draw[-|] (0,0) -- (0,\M) node[anchor=east] {$M$};
        \draw[->] (0,\M) -- (0,\ymax);
        \draw[gray,dashed] (0,\M) -- (\tmax,\M);

        \foreach \tshift in {1.1,.65,.3}
        {
            \draw[domain=0:\tmax,smooth,variable=\t,color=teal] plot ({\t},{\M*tanh(\M*(\t+\tshift))});
        }
        \node[color=teal] at ({\tmax/2},{\ymax/4}) {$\left\|\rho_0\right\|_\infty < M$};

        \foreach \tshift in {1.2,.8,.6}
        {
            \draw[domain=0:\tmax,smooth,variable=\t,color=orange] plot ({\t},{\M*cosh(\M*(\t+\tshift))/sinh(\M*(\t+\tshift))});
        }
        \node[color=orange] at ({\tmax/2},{3*\ymax/4}) {$\left\|\rho_0\right\|_\infty > M$};
    \end{tikzpicture}
    \caption{Evolution of the upper bound on $\left\|\rho_t\right\|_\infty$ in \cref{cor:partial+} for $c_W > 0$ and different values of $\left\|\rho_0\right\|_\infty$.}
    \label{fig:tanhbound}
\end{figure}

\begin{remark}\label{rem:timedependentf}
    If the immigration $f$ is time-dependent with sufficient regularity, 
    then existence and uniqueness for solutions to \cref{eq:classicalpde} still hold, and \cref{thm:partial+} remains true with \cref{eq:partial+} replaced by \[
        \partial_t^+ \left\|\rho_t\right\|_\infty \le \left\|f_t\right\|_\infty - c_W \left\|\rho_t\right\|_\infty^2, \quad t\in [0,T^\star).
    \] In particular, $T^\star = \infty$ and $\sup_{t \ge 0} \left\|\rho_t\right\|_\infty < \infty$ for any initial condition as long as $c_W > 0$ and $\sup_{t \ge 0} \left\|f_t\right\|_\infty < \infty$.
\end{remark}

We now define the index $c_W$. Denote by
\begin{equation}\label{eq:defWbar}
    \left<\nabla W \right> (R) \coloneqq \frac{1}{|\partial B(0,R)|} \int\limits_{\partial B(0,R)} \nabla W \cdot \diff \widehat{n}
\end{equation}
the average radial part of $\nabla W$ at a distance $R > 0$, where $|\partial B(0,R)| = c_d R ^{d-1}$ is the surface area of $B(0,R)$. Recall the definition of the Newtonian potential from \cref{eq:newtonian}. 

\begin{lemma}\label{lem:alphaW}
    The following limits exist,
    \begin{equation}\label{eq:alphaWdef}
        \eta_W \coloneqq \lim_{r \to 0} \frac{\left<\nabla W \right> (r)}{\left<\nabla W_N \right> (r)},\qquad \alpha_W \coloneqq \lim_{R \to \infty} \frac{\left<\nabla W \right> (R)}{\left<\nabla W_N \right> (R)},
    \end{equation}
    with $\eta_W \in \R$ and $\alpha_W \in [0,\infty]$, and
        $\alpha_W  = \eta_W - \int (\Delta W).$
\end{lemma}
The quantities $\eta_W$ and $\alpha_W$ compare the strength of the repulsion to that of the Newtonian potential at short and long range, respectively.
As we will see in \cref{lem:etaW} below, $\eta_W$ also determines the singular behaviour of $W$ at the origin, in that $\Delta W $ contains, in a distributional sense, a multiple $-\eta_W$ of the Dirac delta at zero. Note that \cref{lem:alphaW} implies $\int(\Delta W)_+ < \infty$, because $\eta_W \in \R$ and $\alpha_W \ge 0$. Thus,
\begin{equation}\label{eq:cW}
    c_W \coloneqq \eta_W - \int (\Delta W)_+ 
\end{equation}
is well-defined. If $\alpha_W < \infty$ then $c_W = \alpha_W - \int (\Delta W)_-$ by \cref{lem:alphaW}, so in fact 
\begin{equation}\label{eq:cWsym}
    c_W = \frac{1}{2} \left( \alpha_W + \eta_W - \int |\Delta W| \right),\quad \text{ if $\alpha_W < \infty$.}
\end{equation} 

Our main results apply positively only to potentials with $c_W > 0$. From the different representations of $c_W$ it follows that $\alpha_W > 0$ and $\eta_W > 0$ are both necessary conditions. That is, the critical strength of the repulsion is $|\nabla W(x)| \sim |x|^{1-d}$, and the repulsion potential has to have a (Newtonian) singularity at the origin. The latter is not entirely surprising, since we would not generally expect to be able to find strong bounds on the supremum norm (like \cref{eq:MP,eq:Mproperty,eq:partial+}) of the solution to a PDE with smooth non-local interaction. Before we give examples for $W$, we note that \cref{eq:cW} and linearity of $\eta_W$ in $W$ imply that 
\begin{equation}\label{eq:cWlinear}
    c_{W+W'} \ge c_W + c_{W'}, \qquad c_{aW} = a c_W,
\end{equation}
for any $a > 0$ and potentials $W,W'$ satisfying \cref{ass}. This implies that the class of potentials with $c_W > 0$ is closed under positive linear combinations, as well as small perturbations: If $c_W > 0$ and $\widetilde{W}$ is a perturbative potential, then $c_{W + \varepsilon \widetilde{W}} > 0$ for small $\varepsilon > 0$, in fact as long as $\varepsilon < c_W |c_{\widetilde{W}}|^{-1}$. An example for a perturbation could be a smooth potential that decays faster than Newtonian (so that $\eta_W = \alpha_W = 0$), in which case $c_{\widetilde{W}} = -\frac{1}{2} \int |\Delta \widetilde{W}|$ by \cref{eq:cWsym}.

Examples for interaction potentials that satisfy \cref{ass} include repulsive power laws \[
    W(x) = -P_A(x) \coloneqq -
    \begin{cases}
        \frac{|x|^A}{A} , & A \neq 0,\\
        \log |x| , & A = 0,
    \end{cases}
\] with $1 \ge A \ge 2 - d$ (for \cref{thm:wellposedness} to hold it is sufficient if $A > 1 - d$), and Morse potentials \cite{morse},
\begin{equation}\label{eq:morse}
    W(x) = -C_A \e^{-|x| / \ell_A} + C_R \e^{-|x| / \ell_R},
\end{equation}
with $C_A,C_R,\ell_A,\ell_B > 0$, which are repulsive at long range if $\ell_R < \ell_A$. Of these the most natural example with $c_W > 0$ is the Newtonian potential with $\alpha_W = \eta_W = c_W = 1$, which in this notation is $W_N = -c_d^{-1} P_{2-d}$. Another family of examples is a mixture of repulsive power laws \[
    W(x) =W_N(x) - P_A(x) ,\qquad 1 \ge A > 2 - d,
\] which have $|\nabla W(x)| \sim |x|^{1-d}$ at short and $|\nabla W(x)| \sim |x|^{A-1}$ at long range, and $\alpha_W = \infty$, $\eta_W = c_W = 1$. Morse potentials have $\eta_W = \alpha_W = 0$ and $c_W < 0$, but could be used as a perturbation, see the earlier discussion following \cref{eq:cWlinear}.

\section{Proofs}\label{sec:proofs}


We write $C$ for an unimportant positive constant whose value may change from one appearance to the next. For $\gamma \ge 0$,
\begin{equation}\label{eq:defWgamma}
    \mathcal{X}^\gamma \coloneqq \mathcal{W}^{\gamma,1}\cap \mathcal{W}^{\gamma,\infty},\quad \nnorm{\cdot }_\gamma \coloneqq \left\|\cdot \right\|_{\mathcal{W}^{\gamma,1}} + \left\|\cdot \right\|_{\mathcal{W}^{\gamma,\infty}},
\end{equation}
as well as $\mathcal{X}\coloneqq \mathcal{X}^0 = L^1\cap L^\infty$ and $\nnorm{\cdot }\coloneqq \nnorm{\cdot }_0 = \left\|\cdot \right\|_{L^1}+\left\|\cdot \right\|_\infty$, and $\mathcal{X}_+ \coloneqq \left\{ f\in \mathcal{X}\colon f\ge 0 \right\} $.
If $g\colon \R^d\to \R^m$ for some $m\in \N$, then $\left\|g\right\|_{L^p} \coloneqq \sum_{i=1}^m \left\|g_i\right\|_{L^p}$, similarly for other norms. Recall \cref{eq:defWgamma}, and note that by \cref{lem:embedding}, $\mathcal{X}^\gamma$ with this norm is a Banach space which embeds continuously into $C^{\left\lfloor \gamma \right\rfloor -1,1}(\R^d)$ if $\gamma \ge 1$, and into $C^{\left\lfloor \gamma \right\rfloor , \gamma - \left\lfloor \gamma \right\rfloor }(\R^d)$ if $\gamma \in (0,\infty)\setminus \N$. Here, $C^{k,\beta}(\R^d)$ for $k\in \N_0$ and $\beta \in (0,1]$ is the space of functions $g\colon \R^d\to \R$ which are $k$ times continuously differentiable and for which 
\begin{equation}\label{eq:Ckbeta}
    \left\|g\right\|_{C^{k,\beta}} \coloneqq \sum_{\left| \alpha \right| < k} \left\|\partial^\alpha g\right\|_{L^\infty} + \sum_{\left| \alpha \right| = k} \sup_{x\neq y} \frac{\left| \partial^\alpha g(x) - \partial^\alpha g(y) \right| }{\left| x-y \right|^\beta } < \infty,
\end{equation} with the usual notational conventions for multi-indices $\alpha$. 

\subsection{Existence and Regularity of Solutions}\label{sec:wellposed}
This section contains proofs of local in time well-posedness and regularity of solutions to \cref{eq:classicalpde}. They only require \cref{ass}(i), and \cref{ass}(ii) for $\nabla W$, but no assumptions on $\Delta W$. Note that \cref{ass}(i) becomes $f\in \mathcal{X}^{\gamma_f}$ in the notation \cref{eq:defWgamma}. Many of the ideas in this section were inspired by arguments in Section 2 of \cite{joseheatkernel}.

We begin by establishing well-posedness of a weak version of \cref{eq:classicalpde}. More precisely, by Duhamel's formula we can formally rewrite \cref{eq:classicalpde} with initial condition $\rho_0\in \mathcal{X}$ as an integral equation
\begin{equation}\label{eq:duhamelpde}
    \rho_t = G_t \star \rho_0 + \int_0^t G_s\star f\dd s + \int_0^t \nabla G_{t-s} \star (\rho_s\nabla W\star \rho_s)\dd s,
\end{equation} where $(G_s\colon \R^d \to (0,\infty))_{s > 0}$ denotes the heat kernel, and $F \star H \coloneqq \sum_{i=1}^d F_i \star H_i$ for vector fields $F$ and $H$.

\begin{lemma}\label{lem:W1W2}
    If $\gamma \ge 0$ then there exists $C > 0$ such that for any $g\in \mathcal{X}^\gamma$, \[
    \left\|\nabla W\star g\right\|_{\mathcal{W}^{\gamma,\infty}} \le C \nnorm{g}_\gamma.
    \] In particular, $\nnorm{g \nabla W \star g} \le C \nnorm{g}^2$ 
\end{lemma}
\begin{proof}
    Denote $F_1 \coloneqq \nabla W \ind_{B(0,1)}\in L^1$ and $F_2 \coloneqq \nabla W \ind_{\R^d\setminus B(0,1)}\in L^\infty$. Then, using the fractional version of Young's convolution inequality (see Theorem A.1 in \cite{joseheatkernel}),
    \begin{align*}
        \left\|\nabla W\star g\right\|_{\mathcal{W}^{\gamma,\infty}}
        &\le \left\|F_1 \star g\right\|_{\mathcal{W}^{\gamma,\infty}} + \left\|F_2\star g\right\|_{\mathcal{W}^{\gamma,\infty}}\\
        &\le C( \left\|F_1\right\|_{L^1} \left\|g\right\|_{\mathcal{W}^{\gamma,\infty}} + \left\|F_2\right\|_{L^\infty} \left\|g\right\|_{\mathcal{W}^{\gamma,1}})\\
        &\le C (\left\|F_1\right\|_{L^1} + \left\|F_2\right\|_{L^\infty}) \nnorm{g}_\gamma.
    \end{align*}
    For the additional claim, if $p\in \left\{ 1,\infty \right\} $, then \[
        \left\|g \nabla W \star g\right\|_{L^p} \le \left\|g\right\|_{L^p} \left\|\nabla W \star g\right\|_{L^\infty} \le C \nnorm{g}^2.
    \] 
\end{proof}

We now prove well-posedness of \cref{eq:duhamelpde} locally in time. Standard facts that we will use repeatedly are that, for $\gamma \ge 0$ and $s > 0$,
\begin{equation}\label{eq:GsW}
    \left\|G_s\right\|_{\mathcal{W}^{\gamma,1}} \le C s^{-\gamma / 2},\qquad \left\|\nabla G_s\right\|_{\mathcal{W}^{\gamma,1}} \le C s ^{- (1+\gamma) / 2}.
\end{equation} (See, for example, \cite[p.6]{joseheatkernel}.)

\begin{theorem}[Local in time well-posedness]\label{thm:wellposed}
    Given $\rho_0\in \mathcal{X}$, there is a $T^\star \in (0,\infty]$ and a $\rho\in C([0,T^\star),\mathcal{X})$ which solves \cref{eq:duhamelpde} started at $\rho_0$ and such that any solution $\widetilde{\rho}\in C([0,T),\mathcal{X})$ of \cref{eq:duhamelpde} starting at $\rho_0$ satisfies $T \le T^\star$ and coincides with $\rho$ on $[0,T)$. Furthermore, if $T^\star < \infty$ then $\left\|\rho_t\right\|_\infty\to \infty$ as $t \uparrow T^\star$. If $\rho_0$, $f$, and $W$ are radially symmetric, then so is $\rho_t$ for all $t \in [0,T^\star)$.
\end{theorem}
\begin{proof}
    For fixed $T > 0$ define $\mathcal{F}\colon C([0,T],\mathcal{X}) \to C([0,T],\mathcal{X})$ by the right-hand side (RHS) of \cref{eq:duhamelpde}, that is, \[
    \mathcal{F}[\rho]_t = G_t \star \rho_0+ \int_0^t G_s \star f \dd s + \int_0^t \nabla G_{t-s} \star (\rho_s \nabla W\star \rho_s)\dd s,\quad t \in [0,T].
\] We write $\nnorm{\rho} \coloneqq \sup_{0\le s\le T} \nnorm{\rho_s}$ for $\rho \in C([0,T],\mathcal{X})$. For $p \in \left\{ 1,\infty \right\} $, by Young's convolutional inequality and \cref{lem:W1W2},
    \begin{align*}
        \left\|\mathcal{F}[\rho]_t\right\|_{L^p}
        &\le \left\|G_t\right\|_{L^1} \left\|\rho_0\right\|_{L^p} + \int_0^t \left\|G_s\right\|_{L^1} \left\|f\right\|_{L^p}\dd s \\
        &\qquad + \int_0^t \left\|\nabla G_{t-s}\right\|_{L^1} \left\|\rho_s\right\|_{L^p} \left\|\nabla W\star \rho_s\right\|_{L^\infty}\dd s \\
        &\le \left\|\rho_0\right\|_{L^p} + C\left( t + \int_0^t (t-s)^{-1 / 2} \left\|\rho_s\right\|_{L^p} \nnorm{\rho_s}\dd s\right)\\
        &\le \nnorm{\rho_0} + C\left(t + \sqrt{t} \sup_{s\le t} \nnorm{\rho_s}^2\right).
    \end{align*}
    This implies
    \begin{equation}\label{eq:prflocaltime1}
        \nnorm{\mathcal{F}[\rho]} \le \nnorm{\rho_0} + C \left(T + \sqrt{T} \nnorm{\rho}^2\right).
    \end{equation}
    (In particular, $\mathcal{F}[\rho]_t \in \mathcal{X}$ so $\mathcal{F}$ is well-defined.) By a similar argument, if $\rho^1,\rho^2\in C([0,T],\mathcal{X})$, then
    \begin{align*}
        \big\|\mathcal{F}[\rho^1]_t - \mathcal{F}[\rho^2]_t\big\|_{L^p}
        &\le \int_0^t \left\|\nabla G_{t-s}\right\|_{L^1} \left\|\rho_s^1 \nabla W\star \rho_s^1 - \rho_s^2\nabla W\star\rho_s^2\right\|_{L^p}\dd s\\
        &= \int_0^t \left\|\nabla G_{t-s}\right\|_{L^1} \left\|\rho_s^1 \nabla W\star (\rho_s^1-\rho_s^2) + (\rho_s^1-\rho_s^2) \nabla W\star \rho_s^2\right\|_{L^p}\dd s\\
        &\le C\int_0^t (t-s)^{-1 / 2} \Big[ \left\|\rho_s^1\right\|_{L^p} \nnorm{\rho_s^1-\rho_s^2} + \left\|\rho_s^1-\rho_s^2\right\|_{L^p} \nnorm{\rho_s^2} \Big] \dd s\\
        &\le C \sqrt{t} \sup_{s\le t} \left( \nnorm{\rho^1_s} + \nnorm{\rho^2_s} \right) \nnorm{\rho^1_s-\rho^2_s},
    \end{align*}
    so
    \begin{equation}\label{eq:prflocaltime2}
        \nnorm{\mathcal{F}[\rho^1] - \mathcal{F}[\rho^2]} \le C \sqrt{T} \left( \nnorm{\rho^1}+\nnorm{\rho^2} \right) \nnorm{\rho^1-\rho^2}.
    \end{equation}
    Put $\mathcal{Y} = \left\{ g \in \mathcal{X}\colon \nnorm{g}\le \nnorm{\rho_0}+1 \right\} $. Then by \cref{eq:prflocaltime1,eq:prflocaltime2}, $T > 0$ can be chosen small enough so that $\mathcal{F}$ maps to itself and is a contraction on the space $C([0,T],\mathcal{Y})$. This shows that \[
        T^\star \coloneqq \sup\left\{ t \ge 0\colon \text{there is a solution in $C([0,T],\mathcal{X})$ to \cref{eq:duhamelpde}} \right\} > 0.
        \] Suppose now that $\rho$ and $\widetilde{\rho}$ are two different solutions to \cref{eq:duhamelpde} on $[0,t]$ for some $t \in (0,T^\star]$, starting at $\rho_0$. Let \[
        t_0 \coloneqq \sup \left\{ s \in [0,t]\colon \rho(r) = \widetilde{\rho}(r) \,\forall r\le s \right\} \in [0,t].
        \] Suppose that $t_0 < t$, let $t_1 \coloneqq t_0+\delta \in (t_0,t)$ for some $\delta \in (0,t-t_0)$, and define $\mathcal{F}\colon C([t_0,t_1],\mathcal{X}) \to C([t_0,t_1],\mathcal{X}) $ by  \[
        \mathcal{F}[u]_s = G_{s-t_0}\star \rho(t_0) + \int_{t_0}^{s} G_{s-r}\star f \dd r + \int_{t_0}^s \nabla G_{s-r}\star (u_r \nabla W\star u_r)\dd r,\quad s \in [t_0,t_1].
    \] Put $K\coloneqq \nnorm{\rho}\vee \nnorm{\widetilde{\rho}}$. Similarly to before we can show that $\nnorm{\mathcal{F}[u]_s}\le K + C(\delta+\sqrt{\delta} \nnorm{u}^2)$ for some $C > 0$ and all $s\in[t_0,t_1]$, hence for sufficiently small $\delta > 0$, $\mathcal{F}$ maps to itself and, by an argument identical to that leading to \cref{eq:prflocaltime2}, is a contraction on $C([t_0,t_1],\mathcal{Y})$, where $\mathcal{Y} = \left\{ g\in \mathcal{X}\colon \nnorm{g}\le K+1 \right\} $, so it has a unique fixed point. Since the restrictions of both $\rho$ and $\widetilde{\rho}$ to $[t_0,t_1]$ are fixed points of $\mathcal{F}$, we conclude they must coincide on $[t_0,t_1]$, contradicting the definition of $t_0$.

    We have proved that all solutions must coincide at all times where both are defined, in particular there exists a $\rho \in C([0,T^\star),\nnorm{\cdot })$ which solves \cref{eq:duhamelpde}, and such that any solution $\widetilde{\rho}\in C([0,T),\nnorm{\cdot })$ of \cref{eq:duhamelpde} satisfies $T \le T^\star$ and coincides with $\rho$ on $[0,T)$.

    Now suppose that $T^\star < \infty$, in which case we show that \[
        \nnorm{\rho(t)} \longrightarrow \infty,\quad t \uparrow T^\star.
    \] This already implies that $\left\|\rho(t)\right\|_\infty \to \infty$ as $t\uparrow T^\star$ because $\sup_{t\in [0,T^\star)} \left\|\rho(t)\right\|_{L^1} < \infty$ as a consequence of \cref{lem:L1norm} below. Assume for contradiction that there is a $K > 0$ and a sequence $0 < T_n \uparrow T^\star$ such that $\nnorm{\rho(T_n)} \le K$ for all $n\in \N$. Then we define $\mathcal{F}_n\colon C([T_n,T_n + \delta],\mathcal{X}) \to C([T_n,T_n + \delta],\mathcal{X}) $ by 
        \begin{align*}
            \mathcal{F}_n[u]_t = G_{t-T_n} &\star \rho(T_n) + \int_{T_n}^{t}G_{t-s}\star f\dd s \\
                                           &+ \int_{T_n}^{t} \nabla G_{t-s} \star (u_s \nabla W\star u_s)\dd s,\qquad T_n\le t \le T_n+\delta,
    \end{align*} for some $\delta > 0$. Similarly to before we show $\nnorm{\mathcal{F}_n[u]} \le K + C(\delta + \sqrt{\delta} \nnorm{u}^2)$ for some $C > 0$, and \cref{eq:prflocaltime2} with $T$ replaced by $\delta$. Now choose $\delta$ small enough that $\mathcal{F}_n$ (maps to itself and) is a contraction on the space $C([T_n,T_n+\delta],\mathcal{Y})$ where $\mathcal{Y} = \left\{g\in \mathcal{X}\colon g\colon  \nnorm{g}\le K+1 \right\} $. This choice of $\delta$ can be made independent of $n$, so there exists $n\in \N$ with $T_n + \delta > T^\star $, and we can concatenate $\rho\big\vert_{[0,T_n]} $ with the fixed point of $\mathcal{F}_n$ to obtain a solution to \cref{eq:duhamelpde} defined on $[0,T_n+\delta] \supsetneq [0,T^\star] $, a contradiction. This also implies that, if $T^\star < \infty$, there cannot be a solution defined on $[0,T^\star]$.

If $f$, $W$, and $\rho_0$ are radially symmetric, then $\mathcal{F}$ preserves radial symmetry, so the fixed point iteration started at the constant in time function $(\rho_t \equiv \rho_0\colon t\in [0,T))$ is radially symmetric at every step and converges to $\rho$ uniformly on $\R^d$, so $\rho$ is radially symmetric.
\end{proof}

For the remainder of this section, we assume some $\rho_0\in \mathcal{X}_+$ to be given, and denote by $\rho \in C([0,T^\star),\mathcal{X})$ the unique solution to \cref{eq:duhamelpde} started at $\rho_0$.

\begin{theorem}[Regularity]\label{thm:regularity}
    For any $\gamma\in (0,2+\gamma_f)$, \[
        \rho \in C([0,T^\star),\mathcal{X}_+) \cap C((0,T^\star),\mathcal{X}^\gamma)
    \] 
    In particular, $\rho$ has bounded $C^{2,\gamma_f}(\R^d)$-norm on compact subsets of $(0,T^\star)$, and solves \cref{eq:classicalpde} in the classical sense on $(0,T^\star)$. If $\nnorm{\rho_0}_\gamma < \infty$ (in particular if $\rho_0 \equiv 0$), then $\rho\in C([0,T^\star),\mathcal{X}^\gamma)$. 
\end{theorem}
\begin{proof}
    Fix $T \in (0,T^\star)$. Then $\sup_{0\le s \le T} \nnorm{\rho_s}<\infty$, so it suffices to show that for any $\delta > 0$, $\gamma \in [\gamma_f,2+\gamma_f)$,
    \begin{equation}\label{eq:prfreg:iteration}
        \sup_{\delta \le s \le T} \nnorm{\rho_s}_{\gamma - \gamma_f} < \infty \implies \sup_{2\delta \le s \le T} \nnorm{\rho_s}_\gamma < \infty.
    \end{equation}
    Indeed, by iteration this implies that $\rho \in C((t_0,t_1),\mathcal{X}^\gamma)$ for all $0 < t_0 < t_1 < T^\star$ and hence $\rho\in C((0,T^\star),\mathcal{X}^\gamma)$, for all $\gamma \in [0,2+\gamma_f)$.

    So let $\gamma \in [\gamma_f,2+\gamma_f)$, $\delta > 0$, and assume that $\sup_{\delta \le s \le T} \nnorm{\rho_s}_{\gamma - \gamma_f} < \infty$. Let $p\in \left\{ 1,\infty \right\} $ and $t \in [2\delta, T]$. Then, 
    \begin{align}\label{eq:prfreg:10}
        \left\|\rho_t\right\|_{\mathcal{W}^{\gamma,p}}
        &\le \left\|G_t\right\|_{\mathcal{W}^{\gamma,1}} \left\|\rho_0\right\|_{L^p} + \int_0^t \left\|G_s\right\|_{\mathcal{W}^{\gamma-\gamma_f,1}} \left\|f\right\|_{\mathcal{W}^{\gamma_f,p}} \dd s \\
        &\hspace{2cm} + \int_0^t \left\|\nabla G_{t-s} \star(\rho_s \nabla W \star\rho_s)\right\|_{\mathcal{W}^{\gamma,p}} \dd s.\nonumber
    \end{align}
    By \cref{eq:GsW},
    \begin{align}\label{eq:prfreg:31}
        \left\|G_t\right\|_{\mathcal{W}^{\gamma,1}} \left\|\rho_0\right\|_{L^p} + \int_0^t \left\|G_s\right\|_{\mathcal{W}^{\gamma-\gamma_f,1}} \left\|f\right\|_{\mathcal{W}^{\gamma_f,p}} \dd s 
        &\le C \left(t ^{-\gamma / 2} + \int_0^t s ^{-(\gamma - \gamma_f) / 2} \diff s\right) \\
        &\le C \left( \delta ^{-\gamma / 2} + \int_0^T s ^{-(\gamma - \gamma_f ) / 2} \dd s \right),\nonumber
    \end{align}
    which is a constant independent of $t$ given $\delta,\gamma,T$, and finite because $\gamma-\gamma_f < 2$. If $\gamma = \gamma_f$, in particular $\gamma < 1$, then by Young's fractional convolution inequality and \cref{lem:W1W2},
    \begin{align}
        \int_0^t \left\|\nabla G_{t-s} \star(\rho_s \nabla W \star\rho_s)\right\|_{\mathcal{W}^{\gamma,p}} \dd s 
        &\le \int_0^T \left\|\nabla G_{t-s}\right\|_{\mathcal{W}^{\gamma,1}} \left\|\rho_s \nabla W \star \rho_s\right\|_{L^p}\dd s\\
        &\le C \left(\int_0^T s^{-(1+\gamma) / 2}\diff s\right) \sup_{0\le s \le T} \nnorm{\rho_s}^2,\nonumber
    \end{align} 
    which is a finite constant depending only on $T$ and $\gamma_f$. If $\gamma > \gamma_f$, then we fix $0 < \varepsilon < (1-\gamma_f) \wedge (\gamma-\gamma_f)$, and bound, using Young's fractional convolution inequality,
    \begin{align}\label{eq:prfreg:11}
        \int_0^t \left\|\nabla G_{t-s} \star(\rho_s \nabla W \star\rho_s)\right\|_{\mathcal{W}^{\gamma,p}} \dd s
        &\le \int_0^\delta \left\|\nabla G_{t-s}\right\|_{\mathcal{W}^{\gamma,1}} \left\|\rho_s \nabla W \star \rho_s\right\|_{L^p}\dd s \\
        &\hspace{1cm}+ \int_\delta^t \left\|\nabla G_{t-s}\right\|_{\mathcal{W}^{\gamma_f + \varepsilon,1}} \left\|\rho_s \nabla W \star \rho_s\right\|_{\mathcal{W}^{\gamma-\gamma_f - \varepsilon,p}}\dd s . \nonumber
    \end{align}
    By \cref{lem:W1W2,lem:appalphabeta,lem:appalphamonotone}, putting $\beta \coloneqq \gamma - \gamma_f$,
    \begin{equation}\label{eq:prfreg:12}
        \left\|\rho_s \nabla W \star \rho_s\right\|_{\mathcal{W}^{\beta - \varepsilon,p}}
        \le \left\|\rho_s\right\|_{\mathcal{W}^{\beta,p}} \left\|\nabla W\star\rho_s\right\|_{\mathcal{W}^{\beta,\infty}}
        \le C \left\|\rho_s\right\|_{\mathcal{W}^{\beta,p}} \nnorm{\rho_s}_\beta 
        \le C \nnorm{\rho_s}_\beta^2.
    \end{equation}
    Thus,
    \begin{multline}\label{eq:prfreg:13}
        \int_\delta^t \left\|\nabla G_{t-s}\right\|_{\mathcal{W}^{\gamma_f + \varepsilon,1}}  \left\|\rho_s \nabla W \star \rho_s\right\|_{\mathcal{W} ^{\gamma - \gamma_f - \varepsilon,p}} \dd s\\
        \le C \sup_{\delta \le s < T}\nnorm{\rho_s}_{\gamma - \gamma_f} \int_0^T s ^{-(1+\gamma_f+\varepsilon) / 2} \dd s \le C \sup_{\delta \le s \le T} \nnorm{\rho_s}_{\gamma - \gamma_f}.
    \end{multline}
    By \cref{lem:W1W2},
    \begin{align*}
        \int_0^\delta \left\|\nabla G_{t-s}\right\|_{\mathcal{W}^{\gamma,1}} \left\|\rho_s \nabla W \star \rho_s\right\|_{L^p} \dd s \le C \sup_{0 \le s \le T} \nnorm{\rho_s}^2 \cdot \int_\delta^T s ^{-(1+\gamma) / 2} \dd s < \infty,
    \end{align*}
    which, given $\delta, T, \gamma$, is a constant independent of $t$. Combining this with \cref{eq:prfreg:13,eq:prfreg:11,eq:prfreg:31,eq:prfreg:10} gives
    \begin{align*}
        \left\|\rho_t\right\|_{\mathcal{W}^{\gamma,p}} \le C \left( 1 + \sup_{\delta \le s \le T} \nnorm{\rho_s}_{\gamma - \gamma_f}^2 \right)  < \infty,
    \end{align*}
    for any $t \in [2\delta,T]$ and $p\in \left\{ 1,\infty \right\} $, for $C > 0$ that does not depend on $t$ given $\delta$, $T$, and $\gamma$. We conclude \[
    \sup_{2\delta \le s \le T} \nnorm{\rho_s}_\gamma < \infty
\] as claimed. Then recall from \cref{lem:embedding} that $\mathcal{X}^{2+\gamma_f} \hookrightarrow C^{2,\gamma_f}(\R^d)$.

    If $\gamma\in [0,2+\gamma_f)$ and $\nnorm{\rho_0}_\gamma < \infty$, we can bound $\nnorm{G_t \star \rho_0}_\gamma \le \left\|G_t\right\|_{L^1} \nnorm{\rho_0}_\gamma = \nnorm{\rho_0}_\gamma < \infty$ which is uniform in $t\in [0,T^\star)$, in contrast to the bound $\nnorm{G_t \star \rho_0} \le C t ^{-\gamma / 2}$ that we used in \cref{eq:prfreg:10}. This lets us prove \cref{eq:prfreg:iteration} with $\delta = 0$, giving $\rho\in C([0,T^\star),\mathcal{X}^\gamma)$.

    We now show that $\rho$ solves \cref{eq:classicalpde} in the classical sense on $(0,T^\star)$. Clearly $\rho(0,\cdot ) = \rho_0$. It suffices to show now that \cref{eq:classicalpde} is satisfied at a fixed $t\in (0,T^\star)$ and $x\in \R^d$. Because of the instant regularisation of $\rho$, we can assume without loss of generality that already $\rho_0\in \mathcal{X}^\gamma$ for all $\gamma \in [0,2+\gamma_f)$, so that $\sup_{0 \le s \le t} \nnorm{\rho_s}_\gamma < \infty$ by the above, and thus 
    \begin{equation}\label{eq:prfreg:30}
        \sup_{0 \le s \le t} \nnorm{\rho_s \nabla W \star\rho_s}_\gamma < \infty,\quad \gamma \in [0,2+\gamma_f)
    \end{equation} by \cref{lem:W1W2,lem:appalphabeta}. In particular $\rho_s \nabla W \star\rho_s \in C^{2}(\R^d)$ for all $s \in [0,t]$, so we can rewrite \cref{eq:duhamelpde} as \[
        \rho(t) = G_t \star \rho_0 + \int_0^t G_{t-s} \star f \dd s + \int_0^t G_{t-s} \star \nabla \cdot (\rho_s \nabla W\star \rho_s)\dd s.
    \] Recall that $G\colon [0,t] \times \R^d \to \R$ is smooth with globally bounded derivatives of any order, and $\partial_t G_t = \frac{1}{2}\Delta G_t$. To show that we can pull the time derivative into the integrals, we establish the following bounds.
    \begin{align*}
        \left\|(\partial_t G_{t-s}) \star f\right\|_\infty 
        = \frac{1}{2} \left\| (\Delta G_{t-s}) \star f\right\|_\infty
        &\le C \left\| G_{t-s} \star f\right\|_{\mathcal{W}^{2,\infty}}\\
        &\le C \left\|G_{t-s}\right\|_{\mathcal{W}^{2-\gamma_f,1}} \left\|f\right\|_{\mathcal{W}^{\gamma_f,\infty}}\\
        &\le C (t-s)^{-(2-\gamma_f) / 2} \nnorm{f}_{\gamma_f},
    \end{align*}
    which is integrable over $(0,t)$. Furthermore,
    \begin{align*}
        \left\|(\partial_t G_{t-s}) \star \nabla \cdot (\rho_s \nabla W \star \rho_s)\right\|_\infty
        &= \frac{1}{2} \left\|(\Delta G_{t-s}) \star \nabla \cdot (\rho_s \nabla W \star \rho_s)\right\|_\infty\\
        &\le C \left\|G_{t-s} \star \rho_s \nabla W \star \rho_s\right\|_{\mathcal{W}^{3,\infty}}\\
        &\le C \left\|G_{t-s}\right\|_{\mathcal{W}^{1,1}} \left\|\rho_s \nabla W \star \rho_s\right\|_{\mathcal{W}^{2,\infty}}\\
        &\le C (t-s) ^{- 1 / 2} \left( \sup_{0 \le s \le t}\nnorm{\rho_s}_{2}^2\right),
    \end{align*}
    where we recalled \cref{eq:prfreg:30,lem:W1W2} in the final step, which is also integrable over $(0,t)$. Hence, $\rho(\cdot ,x)$ is differentiable at $t$ and 
    \begin{align*}
    \partial_t \rho_t
    &= (\partial_t G_t) \star\rho_0 + f + \int_0^t (\partial_t G_{t-s}) \star f\dd s + \nabla \cdot (\rho_t\nabla W\star \rho_t) \\
    &\hspace{2cm} + \int_0^t (\partial_t G_{t-s}) \star \nabla \cdot (\rho_s \nabla W\star \rho_s)\dd s\\
    &= \frac{1}{2}\Delta \rho_t + \nabla \cdot (\rho_t\nabla W\star \rho_t) + f.
    \end{align*}
    Non-negativity will be proved in the following lemma.
\end{proof}

\begin{lemma}\label{lem:L1norm}
    For any $t \in [0,T^\star)$, $\rho_t$ is non-negative and $\left\|\rho_t\right\|_{L^1} = \left\|\rho_0\right\|_{L^1} +  t \left\|f\right\|_{L^1}$.
\end{lemma}
\begin{proof}
    The idea is to show that $\int \rho(t,x)_- \dd x = 0$ for all $t \ge 0$. For that purpose, let $(j_\varepsilon)_{\varepsilon > 0}$ be a family of smooth and convex functions such that $j_\varepsilon(s) = (-s)\vee 0$ on $\R \setminus [-\varepsilon,0]$, and $0\le j_\varepsilon'' \le 2 / \varepsilon$ in $[-\varepsilon,0]$. Then for any $\varepsilon > 0$,
    \begin{align}\begin{split}\label{eq:jeps}
        \frac{\dd }{\dd t} \int j_\varepsilon(\rho(t,x)) \dd x
        &= \int_{\R^d} j_\varepsilon'(\rho(t,x)) \left( \frac{1}{2}\Delta \rho_t + \nabla \cdot (\rho_t \nabla W\star \rho_t) + f \right)(x) \dd x\\
        &\le - \int_{\R^d} j_\varepsilon''(\rho(t,x)) \nabla \rho_t(x)\cdot  \left( \frac{1}{2}\nabla \rho_t + \rho_t \nabla W\star \rho_t \right)(x) \dd x\\
        &\le - \int_{\R^d} J_\varepsilon(\rho(t,x)) \nabla \rho_t(x) \cdot (\nabla W\star \rho_t)(x) \dd x\\
        &\eqqcolon H_\varepsilon(t),
    \end{split}\end{align}
    where in the second step we used that $j_\varepsilon' \le 0$ and $f \ge 0$, and in the third step we used that $j_\varepsilon''(\rho) \left| \nabla \rho \right| ^2 \ge 0$ and put $J_\varepsilon(s) \coloneqq j_\varepsilon''(s) s$. To understand why this was helpful, formally replace $j_\varepsilon(s)$ by $j(s)\coloneqq (-s)\vee 0$, so that the left-hand side (LHS) turns into $\frac{\dd }{\dd t}\int \rho(t,x)_- \dd x$, and the RHS is zero because $j''(x) = \delta(x)$, so $J(\rho) = \delta(\rho)\rho = 0$.

    More precisely, \cref{eq:jeps} implies $\int j_\varepsilon(\rho(t,x))\dd x \le j_\varepsilon(\rho_0(x)) + \int_0^t H_\varepsilon(s)\dd s$ for all $\varepsilon > 0$ and $t\in [0,T^\star)$. The LHS converges by dominated convergence to $\int \rho(t,x)_- \dd x$, and the RHS converges by dominated convergence to zero. Indeed, $J_\varepsilon \to 0$ pointwise on $\R\setminus \left\{ 0 \right\} $ as $\varepsilon \to 0$, and $0 \le J_\varepsilon(s) \le 2$ for all $s\in \R$ and $\varepsilon > 0$ by assumption on $j_\varepsilon$. Hence, $\int \rho(t,x)_- \dd x = 0$ for all $t \in [0,T^\star)$, so by continuity of $\rho(t)$, it is non-negative for all $t\in [0,T^\star)$.

    For the second claim,
    \begin{align*}
        \int \rho(t,x) \dd x
        &= \int_{\R^d}\int_{\R^d} G_t(x-y) \rho_0(y) \dd y\dd x+ \int_0^t \int_{\R^d}\int_{\R^d} G_{s}(x-y) f(y)\dd y\dd x\dd s \\
        &\qquad\qquad + \int_0^t \int \int (\nabla G_{t-s})(x-y) (\rho_s\nabla W\star \rho_s)(y)\dd y\dd x\dd s\\
        &= \int \rho_0(x) \dd x +  t \left( \int_{\R^d}f(x)\dd x \right),
    \end{align*}
    where we used that $\int G_s = 1$ and $\int \nabla G_s = 0$ for all $s > 0$.
\end{proof}

\subsection{Maximum Principle and Global Boundedness}\label{sec:MPproof}

The main goal of this section is to prove \cref{thm:partial+,thm:MP,thm:Mproperty}. For clarity of presentation, we will postpone the proofs of some auxiliary lemmas until after we show \cref{thm:partial+,thm:MP,thm:Mproperty}.

The representation of $c_W$ used mainly in the proofs is $c_W = \eta_W - \int (\Delta W)_+$, and a crucial first step is to establish the following meaning of $\eta_W$.

\begin{lemma}\label{lem:etaW}
    There is $\xi_W \in \R^d$ such that for any $g\in \mathcal{X}^1\cap C^1(\R^d)$ (in particular any $g\in \mathcal{X}^{\gamma}$ with $\gamma > 1$), \[
        \nabla \cdot (\nabla W\star g) = (\Delta W)\star g -\eta_W g + \xi_W \cdot \nabla g.
        \] 
    If $W(x) = W(-x)$ for all $x\in \R^d \setminus \left\{ 0 \right\} $, then $\xi_W = 0$.
\end{lemma}
That is, $\Delta W$ can be viewed in a distributional sense as a sum of the function $x \mapsto \Delta W(x)$ on $\R^d\setminus \left\{ 0 \right\} $, and the distribution $-\eta_W \delta_0 - \xi_W \cdot \nabla \delta_0$, where $\delta_0$ denotes the Dirac mass at the origin. The significance of this lies in the following application: If $g$ is a non-negative function with a global maximum at some $x_0$, then $\nabla g(x_0) = 0$ and $g(x_0) = \left\|g\right\|_\infty$, so
\begin{align*}
    \nabla \cdot (\nabla W \star g)(x_0) 
    &= (\Delta W \star g)(x_0) - \eta_W g(x_0) \\
    &\le \left\|g\right\|_\infty \int (\Delta W)_+  - \eta_W g(x_0) \\
    &= -c_W g(x_0).
\end{align*}
This already proves the first assertion of the following lemma. Denote by $C_{c+}^\infty(\R^d)$ the set of infinitely differentiable, compactly supported, non-negative functions on $\R^d$.

\begin{lemma}\label{lem:divbound}
    If $g \in \mathcal{X}^1\cap C^1(\R^d)$ is non-negative and has a global maximum at $x_0\in \R^d$, then 
    \begin{equation}\label{eq:divbound}
        \nabla \cdot (\nabla W \star g)(x_0) \le - c_W g(x_0).
    \end{equation}
    For any $c > c_W$ and $x_0\in \R^d$, there exists $g \in C_{c+}^\infty(\R^d)$ such that $g$ has a global maximum of any given height at $x_0$, $\Delta g(x_0) = 0$, and \[
    \nabla \cdot (\nabla W \star g)(x_0) > -c g(x_0).
    \]
\end{lemma}
Note that if $W = W_N$ is the Newtonian potential, then $\nabla \cdot (\nabla W \star g) = -g$ which trivially implies the assertion of \cref{lem:divbound}. Now if $t > 0$ and $\rho_t$ has a global maximum at some $x_0\in \R^d$, then $\Delta \rho_t(x_0) \le 0$, $\nabla \rho_t(x_0) = 0$, and $\rho_t(x_0) = \left\|\rho_t\right\|_\infty$, so evaluating the PDE \cref{eq:classicalpde} at $(t,x_0)$ and using \cref{lem:etaW} gives
\begin{align}\begin{split}\label{eq:dtmaxbound}
    \partial_t \rho_t(x_0) 
     &\le \rho_t(x_0) \nabla \cdot (\nabla W \star \rho_t)(x_0) + f(x_0)\\
    &\le \left\|f\right\|_\infty - c_W \left\|\rho_t\right\|_\infty^2.
\end{split}\end{align}
This is almost the differential inequality in \cref{thm:partial+}, where the LHS is $\partial^+_t\left\|\rho_t\right\|_\infty$, and the following two lemmas close that gap.

\begin{lemma}\label{lem:partial+}
    Suppose $T > 0$ and $g\in C([0,T],\mathcal{X}_+)$ such that $g(t,\cdot )$ is Lipschitz continuous for all $t \in (0,T]$, and $g$ is differentiable in time on $(0,T]$ with $\partial_t g\colon (0,T] \times \R^d \to \R $ jointly continuous. Suppose further that there is a continuous function $C\colon [0,T] \to [0,\infty)$ such that for all $t \in (0,T]$, if $g_t$ has a global maximum at $x$, \[
        \partial_t g(t,x) \le C(t).
    \] Then for all $t \in [0,T)$, \[
    \partial_t^+ \left\|g_t\right\|_\infty \le C(t).
    \] 
\end{lemma}
\begin{lemma}\label{lem:rhoreg}
    For any $\rho_0\in \mathcal{X}_+$ and $T \in (0,T^\star)$, $\rho\colon [0,T]\times \R^d \to [0,\infty)$ satisfies the assumptions of \cref{lem:partial+}.
\end{lemma}

We can now prove \cref{thm:partial+}.

\begin{proof}[Proof of \cref{thm:partial+}] 
    The differential inequality \cref{eq:partial+} follows from \cref{eq:dtmaxbound,lem:partial+,lem:rhoreg}, so it remains to prove sharpness. Let $c > c_W$ and $h > 0$, put $c' \coloneqq (c+c_W) / 2 > c_W$, and choose $x_0\in \R^d$ with $f(x_0) \ge \left\|f\right\|_\infty - (c-c_W) h^2 / 2$. Then by \cref{lem:divbound}, there is $\rho_0\in C_{c+}^\infty(\R^d)$ such that $\rho_0$ has a global maximum of height $h$ at $x_0$, $\Delta \rho_0(x_0) = 0$, $\rho_0(x_0) = \left\|\rho_0\right\|_\infty  = h$, and $\nabla \cdot (\nabla W \star \rho_0)(x_0) > -c'\rho_0(x_0) = -c' h$. Then, since $\nabla \rho_0(x_0) = 0$, 
    \begin{align*}
        \partial_t \rho(t,x_0)\Big\vert_{t=0} 
        &= \rho_0(x_0) \nabla \cdot (\nabla W\star \rho_0)(x_0) + f(x_0) \\
        &> -c'h^2 + \left\|f\right\|_\infty - (c-c_W)h^2 / 2 \\
        &= \left\|f\right\|_\infty - c \left\|\rho_0\right\|_\infty^2.
    \end{align*} 
    In particular, 
    \begin{align*}
    \partial_t^+ \left\|\rho_t\right\|_\infty \Big\vert_{t = 0} 
    = \varlimsup_{t \downarrow 0} \frac{\left\|\rho_t\right\|_\infty - \left\|\rho_0\right\|_\infty}{t} 
    \ge \varlimsup_{t\downarrow 0} \frac{\rho(t,x_0) - \rho_0(x_0)}{t} 
    &= \partial_t \rho(t,x_0)\Big\vert_{t = 0} \\
    &> \left\|f\right\|_\infty - c \left\|\rho_0\right\|_\infty^2.
    \end{align*}
\end{proof}

As we mentioned in \cref{sec:results}, the differential inequality \cref{eq:partial+} can be turned into an explicit upper bound, recall also \cref{fig:tanhbound} for an illustration.

\begin{corollary}\label{cor:partial+}
   If $c_W > 0$ and $\rho_0\in \mathcal{X}_+$, then $T^\star = \infty$ and for all $t\ge 0$,
    \begin{equation}\label{eq:tanhbound}
       \left\|\rho_t\right\|_\infty \le
       M \begin{cases}
           \coth(Mc_W(t+t_0)) ,& \left\|\rho_0\right\|_\infty > M,\\
           \tanh(Mc_W(t+t_0)) ,& \left\|\rho_0\right\|_\infty < M,\\
           1 ,& \left\|\rho_0\right\|_\infty = M,
       \end{cases}
   \end{equation}
   where $M = \sqrt{\left\|f\right\|_\infty / c_W} $, and $t_0\in \R$ is such that the right-hand side at $t = 0$ is $\left\|\rho_0\right\|_\infty$.
\end{corollary}
\begin{proof}
    Suppose $c_W > 0$ and let $M\coloneqq \sqrt{\left\|f\right\|_\infty / c_W} $. Then the unique solution to the ODE \[
    \begin{cases}
        G'(t) = \left\|f\right\|_\infty - c_W G(t)^2 ,\quad t \ge 0,\\
        G(0) = \left\|\rho_0\right\|_\infty,
    \end{cases}
    \] is given by \[
    G(t) = M \begin{cases}
        \coth \left( c_W M(t+t_0) \right) ,& G(0) > M,\\
        \tanh \left( c_W M(t+t_0) \right) ,& G(0) < M,\\
        1 ,& G(0) = M,
    \end{cases}
\] where $t_0\in \R$ is such that $G(0) = \left\|\rho_0\right\|_\infty$. Then \cref{eq:partial+} implies $\partial^+_t \left\|\rho_t\right\|_\infty \le G'(t)$ for all $t\in [0,T^\star)$ and thus $\left\|\rho_t\right\|_\infty \le G(t)$ for all $t\in [0,T^\star)$. In particular, if $T^\star$ were finite then $\left\|\rho_t\right\|_\infty \le G(t) \to G(T^\star) < \infty$ as $t \uparrow T^\star$, contradicting \cref{thm:wellposed}.
\end{proof}

This also proves \cref{prop:globalexistence}.
The positive assertions of \cref{thm:MP,thm:Mproperty} are now straightforward consequences of \cref{cor:partial+}, and the sharpness follows from \cref{thm:partial+}.

\begin{proof}[Proof of \cref{thm:MP,thm:Mproperty}]
    Suppose that $c_W > 0$, put $M \coloneqq  \sqrt{\left\|f\right\|_\infty / c_W} $, and let $\rho_0\in \mathcal{X}_+$. If $\left\|\rho_0\right\|_\infty = M$, then \cref{eq:tanhbound} implies $\left\|\rho_t\right\|_\infty \le M = \left\|\rho_0\right\|_\infty$ for all $t\in [0,T^\star)$. If $\left\|\rho_0\right\|_\infty < M$, then \cref{eq:tanhbound} implies \[
        \left\|\rho_t\right\|_\infty \le M \tanh(c_W M(t+t_0)) < M,
    \] for all $t\in [0,T^\star)$. If $\left\|\rho_0\right\|_\infty > M$, then \[
        \left\|\rho_t\right\|_\infty \le M \coth(c_W(t+t_0)) \eqqcolon G(t),
    \] for all $t\in [0,T^\star)$. Since $G(0) > M$, we must have $t_0 > 0$. Thus, $G$ is strictly decreasing on $[0,\infty)$, so \[
        \left\|\rho_t\right\|_\infty \le G(t) < G(0) = \left\|\rho_0\right\|_\infty,\quad t\in (0,T^\star).
    \] 
    For the sharpness statements, suppose first that $c_W > 0$ and $M < \sqrt{\left\|f\right\|_\infty / c_W} $. Then it suffices to show that there exists $\rho_0 \in \mathcal{X}_+$ with $\left\|\rho_0\right\|_\infty = M$ and $\partial_t^+ \left\|\rho_t\right\|_\infty \Big\vert_{t=0} > 0$. For that purpose note that $\left\|f\right\|_\infty - c_W M^2 > 0$, so there is $c > c_W$ such that still $\left\|f\right\|_\infty - c M^2 > 0$. Then there exists by \cref{thm:partial+} a $\rho_0\in C_{c+}^\infty(\R^d)$ with $\left\|\rho_0\right\|_\infty  = M$ and
    \begin{align*}
        \partial^+_t \left\|\rho_t\right\|_\infty\Big\vert_{t=0} > \left\|f\right\|_\infty - c \left\|\rho_0\right\|_\infty^2 = \left\|f\right\|_\infty - c M^2 > 0.
    \end{align*}
    Now suppose that $c_W < 0$ and $M > 0$. Then again it suffices to show that there exists $\rho_0\in \mathcal{X}_+$ with $\left\|\rho_0\right\|_\infty = M$ and $\partial_t^+ \left\|\rho_t\right\|_\infty > 0$. Let $c_W < c < 0$, then by \cref{thm:partial+} there exists $\rho_0\in C_{c+}^\infty(\R^d)$ with $\left\|\rho_0\right\|_\infty = M$ and
    \begin{align*}
        \partial^+_t \left\|\rho_t\right\|_\infty\Big\vert_{t=0} > \left\|f\right\|_\infty - c \left\|\rho_0\right\|_\infty^2 \ge \left\|f\right\|_\infty \ge 0.
    \end{align*}
\end{proof}

In the following subsections, we prove \cref{lem:etaW,lem:divbound,lem:partial+,lem:rhoreg}, as well as \cref{lem:alphaW} (stated in \cref{sec:results}) which establishes existence of the limits defining $\eta_W$ and $\alpha_W$. Finally, we prove the claim made in the introduction regarding the asymptotics of \cref{eq:classicalpde} with $W \equiv 0$.

\subsubsection{Proofs of \cref{lem:alphaW,lem:etaW}}

Using the same argument as in the proof of \cref{lem:W1W2}, we can show that \cref{ass}(ii) implies
\begin{equation}\label{eq:LWconv}
    \left\|\Delta W \star g\right\|_{\mathcal{W}^{\gamma,\infty}} \le C \nnorm{g}_\gamma
\end{equation}
for any $\gamma \ge 0$ and $g\in \mathcal{X}^\gamma$ (where $C > 0$ depends on $\gamma$ but not $g$). We further need the following technical lemma.
\begin{lemma}\label{lem:rhotilde}
    There is $C > 0$ such that for any $g\in \mathcal{W}^{1,\infty}$ with $g(0) = 0$ and any $\varepsilon > 0$, there exists $\widetilde{g} \in \mathcal{W}^{1,\infty}$ with $\widetilde{g} = g$ on $B(0,\varepsilon)$, $\widetilde{g} \equiv 0$ on $\R^d\setminus B(0,2\varepsilon)$, and $\left\|\widetilde{g}\right\|_{\mathcal{W}^{1,\infty}} \le C \left\|g\right\|_{\mathcal{W}^{1,\infty}(B(0,\varepsilon))}$.
\end{lemma}
\begin{proof}
    Define $\widetilde{g}$ in radial coordinates by \[
        \widetilde{g}(r,\varphi) \coloneqq \begin{cases}
            g(r,\varphi) ,& r < \varepsilon,\\
            g(2\varepsilon - r,\varphi) ,& \varepsilon\le r < 2\varepsilon,\\
            0 ,& 2\varepsilon \le r,
        \end{cases}
    \] where $\varphi$ stands collectively for all $d-1$ angular variables. Then $\widetilde{g}$ is continuous, $\left\|\widetilde{g}\right\|_{L^\infty} \le \left\|g\right\|_{L^\infty(B(0,\varepsilon))}$ and $\left\|\widetilde{g}\right\|_{\text{Lip}} \le \left\|g\right\|_{\text{Lip}(B(0,\varepsilon))}$, where \[
    \left\|h\right\|_{\text{Lip}(\Omega)} \coloneqq \sup_{\substack{x,y\in \Omega \\ x\neq y} }  \frac{\left|h(x)-h(y)\right|}{\left| x-y \right| }
\] for $h\colon \R^d \to \R$ and $\Omega \subset \R^d$. Then recall the well-known result that $\mathcal{W}^{1,\infty} = C^{0,1}(\R^d)$, the space of bounded Lipschitz functions.
\end{proof}

The following lemma is identical to \cref{lem:etaW} except that $\eta_W$ is replaced by a possibly different value $\eta_W'$. We then show together with \cref{lem:alphaW} (which proves existence of $\alpha_W$ and $\eta_W$) that $\eta_W = \eta_W'$.
\begin{lemma}\label{lem:etaW2}
    There exist $\eta_W'\in \R$ and $\xi_W \in \R^d$ such that for any $g\in \mathcal{X}^1\cap C^1(\R^d)$ (in particular any $g\in \mathcal{X}^{\gamma}$ with $\gamma > 1$), \[
        \nabla \cdot (\nabla W\star g) = (\Delta W)\star g -\eta_W' g + \xi_W \cdot \nabla g.
        \] 
    If $W(x) = W(-x)$ for all $x\in \R^d \setminus \left\{ 0 \right\} $, then $\xi_W = 0$.
\end{lemma}
\begin{proof}
    We prove the claim at a fixed but arbitrary $x_0\in \R^d$, and assume without loss of generality that $x_0 = 0$. Recall that $\nabla W\star g \in \mathcal{W}^{1,\infty}$ for any $g\in \mathcal{X}^1$ by \cref{lem:W1W2}, so it is bounded and globally Lipschitz continuous. This remains true if we replace $g$ by $g \varphi_\varepsilon$ for a smooth radially symmetric bump function $\varphi_\varepsilon$ with $ \ind_{B(0,\varepsilon)} \le \varphi_\varepsilon \le \ind_{B(0,2\varepsilon)}$. Put $\psi_\varepsilon \coloneqq 1-\varphi_\varepsilon$. Then,
    \begin{align*}
        \nabla \cdot (\nabla W\star g)
        &= \nabla \cdot (\nabla W\star g\psi_\varepsilon) + \nabla \cdot (\nabla W\star g \varphi_\varepsilon).
    \end{align*}
    Now, \[
        \nabla \cdot (\nabla W\star g \psi_\varepsilon)(0) = \nabla \cdot \left(\int \nabla W(\cdot -y) g(y) \psi_\varepsilon(y)\dd y \right)(0)= (\Delta W \star g\psi_\varepsilon)(0),
    \] because $\Delta W$ is bounded away from zero. More precisely, as long as $x\in B(0,\varepsilon / 2)$, say, then $ \left| x-y \right| \ge \left| y \right|  - \left| x \right| \ge \varepsilon / 2  $ for any $y\in \R^d$ with $\psi_\varepsilon(y) > 0$, and $\sup_{B(0,\varepsilon / 2)^{c}} \Delta W < \infty$ by \cref{ass}(ii). By \cref{eq:LWconv}, $\Delta W \star g$ is well-defined, hence by dominated convergence, $(\Delta W\star g \psi_\varepsilon)(0) \to (\Delta W\star g)(0)$ as $\varepsilon \to 0$. Thus,
    \[
        \mathcal{F}[g] \coloneqq \lim_{\varepsilon \to 0} \nabla \cdot (\nabla W\star g \varphi_\varepsilon ) (0) = \nabla \cdot (\nabla W\star g)(0) - (\Delta W\star g)(0)
        \] exists and is finite for any $g\in \mathcal{X}^1$. Inspecting the RHS above we see that \[
        \mathcal{F} \text{ is linear and } \left| \mathcal{F}[g] \right| \le C \nnorm{g}_{1}.
    \] Indeed, $\left| \mathcal{F}[g] \right| \le \left\|\nabla \cdot (\nabla W\star g)\right\|_{L^\infty} + \left\|\Delta W \star g\right\|_{L^\infty}$, and
        \begin{align*}
            \left\|\nabla \cdot (\nabla W\star g)\right\|_{L^\infty}
            &\le \left\|\nabla W\star g\right\|_{\mathcal{W}^{1,\infty}} \le C \nnorm{g}_1
        \end{align*}
        by \cref{lem:W1W2}, and $\left\|\Delta W\star g\right\|_{L^\infty} \le C \nnorm{g}_1$ by \cref{eq:LWconv}. Furthermore, from the definition it is immediate that $\mathcal{F}[g]$ depends on $g$ only locally around $0$, in the sense that $\mathcal{F}[g_1] = \mathcal{F}[g_2]$ if $g_1, g_2 \in \mathcal{X}^1$ coincide in a neighbourhood of $0$.

        Now suppose that $g\in \mathcal{X}^1$ with $g(0) = 0$. Let $\varepsilon\in (0,1)$, then by \cref{lem:rhotilde} there is $\widetilde{g}\in \mathcal{X}^1$ such that $\widetilde{g} = g$ on $B(0,\varepsilon)$, $\widetilde{g} \equiv 0$ on $\R^d\setminus B(0,2\varepsilon)$ and $\nnorm{\widetilde{g}}_1 \le C \left\|\widetilde{g}\right\|_{\mathcal{W}^{1,\infty}} \le C \left\|g\right\|_{\mathcal{W}^{1,\infty}(B(0,\varepsilon))}$ with constants independent of $g$ and $\varepsilon$ (the first inequality holds because $\widetilde{\rho}$ is supported in $B(0,\varepsilon) \subset B(0,1)$). This implies that $\left| \mathcal{F}[g] \right| = \left| \mathcal{F}[\widetilde{g}] \right| \le C \nnorm{\widetilde{g}}_1 \le C \left\|g\right\|_{\mathcal{W}^{1,\infty}(B(0,\varepsilon))}$. Letting $\varepsilon \to 0$ implies \[
            \forall g\in \mathcal{X}^1, g(0) = 0 \colon \left| \mathcal{F}[g] \right| \le C  \varliminf_{\varepsilon\to 0} \left\|g\right\|_{\mathcal{W}^{1,\infty}(B(0,\varepsilon))}.
            \] In particular, if $g\in \mathcal{X}^1\cap C^1(\R^d)$ and $g(0) = \nabla g(0) = 0$ then $\mathcal{F}[g] = 0$. By linearity, there must be $\eta_W' \in \R$ and $\xi_W \in \R^d$ with \[
            \mathcal{F}[g] = -\eta_W' g(0) + \xi_W \cdot \nabla g(0)
        \] for all $g\in \mathcal{X}^1 \cap C^1(\R^d)$. Now suppose that $W(x)= W(-x)$ for all $x\in \R^d$. To show that $\xi_W = 0$ it suffices to prove that whenever $g\in \mathcal{X}^1\cap C^1(\R^d)$ with $g(0) = 0$, then $\mathcal{F}[g] = 0$. Since we already know that $\mathcal{F}[g]$ does not change when we replace $g$ by a function in $\mathcal{X}^1\cap C^1(\R^d)$ that has the same value and gradient at $0$, we may assume that $g(x) = \beta \cdot x$ for all $x\in B(0,1)$, where $\beta = \nabla g(0)$. Then,
        \begin{multline*}
            \nabla \cdot (\nabla W\star g \varphi_\varepsilon) (0)
            = \beta \cdot (\nabla W \star \varphi_\varepsilon)(0) + (\nabla W \star g \nabla \varphi_\varepsilon)(0)\\
            = \beta \cdot \int \nabla W(x) \varphi_\varepsilon(x) \dd x + \beta \cdot \int x (\nabla W(x) \cdot \nabla \varphi_\varepsilon(x))\dd x
            = 0,
        \end{multline*}
        because both integrands are antisymmetric, hence $\mathcal{F}[g] = 0$. 
\end{proof}

We now prove \cref{lem:alphaW}, which establishes existence of $\eta_W$ and $\alpha_W$, and show that $\eta_W = \eta_W'$. Note that $\left<\nabla W_N \right> (r) = - c_d ^{-1} r ^{1 - d}$, so equivalently to \cref{eq:alphaWdef} we could write \[
    \eta_W = -c_d \lim_{r \to 0} r^{d - 1}\left<\nabla W \right> (r),\qquad \alpha_W = -c_d \lim_{R \to \infty} R^{d-1}\left<\nabla W \right> (R).
\]

\begin{proof}[Proof of \cref{lem:alphaW}]
    Let $(g_\varepsilon)_{\varepsilon > 0}$ be a smooth approximation to unity with $g_\varepsilon$ supported in $B(0,\varepsilon)$ for all $\varepsilon > 0$. Then by \cref{lem:etaW2,eq:defWbar}, and recalling that $c_d > 0$ denotes the surface area of the unit ball in $\R^d$,
    \begin{align}\begin{split}\label{eq:prfWrel:11}
        \int\limits_{\partial B(0,R)} (\nabla W \star g_\varepsilon) \cdot \dd \widehat{n}
        &= \int_{B(0,R)} \nabla \cdot (\nabla W\star g_\varepsilon)(x)\dd x\\
        &= -\eta_W' \int_{B(0,R)}g_\varepsilon(x) \dd x + \xi_W \cdot \int_{B(0,R)} \nabla g_\varepsilon(x) \diff x  \\
        &\hspace{2cm} + \int_{B(0,R)} (\Delta W \star g_\varepsilon)(x)\dd x\\
        &= -\eta_W' + \int_{B(0,R)} (\Delta W \star g_\varepsilon)(x) \diff x,
    \end{split}\end{align}
    where we used that $\int_{B(0,R)} \nabla g_\varepsilon(x) \diff x = 0$ by integration by parts. We want to take $\varepsilon \to 0$ for fixed $R > 0$. On the LHS we obtain \[
        \int\limits_{\partial B(0,R)} (\nabla W \star g) \cdot \diff \widehat{n} \longrightarrow  \int\limits_{\partial B(0,R)} \nabla W \cdot \diff \widehat{n} = c_d R ^{d-1} \left<\nabla W \right> (R),
    \] by dominated convergence, which is applicable because $\nabla W$ is bounded away from the origin by \cref{ass}(ii). For the RHS, we cannot directly apply dominated or monotone convergence without further assumptions on $\Delta W$. Instead, we write $\Delta W \star g_\varepsilon = (\Delta W)_+ \star g_\varepsilon - (\Delta W)_- \star g_\varepsilon$, and 
    \begin{align}\begin{split}\label{eq:prfWrel:10}
        \int_{B(0,R)} [(\Delta W)_+ \star g_\varepsilon] (x) \diff x
        &= \int_{B(0,R)} \int_{B(0,\varepsilon)} (\Delta W)_+(x-y) g_\varepsilon(y) \diff y \diff x\\
        &= \int_{B(0,\varepsilon)}g_\varepsilon(y) \left( \int_{B(y,R)} (\Delta W)_+(x) \diff x\right) \diff y.
    \end{split}\end{align}
    Now, for any $y\in B(0,\varepsilon)$, if say $\varepsilon < R  /2$,
    \begin{align*}
        \left| \int_{B(0,R)} \Delta W(x)_+ \diff x - \int_{B(y,R)} \Delta W(x)_+ \diff x \right| 
        &\le \int\limits_{B(0,R) \Delta B(y,R)} |\Delta W(x)| \diff x\\
        &\le C \varepsilon R ^{d-1} \sup_{|z| \ge R / 2}|\Delta W(z)|,
    \end{align*}
    where $A \Delta B = (A\setminus B) \cup (B \setminus A)$ for $A,B\subset \R^d$. This goes to zero as $\varepsilon \to 0$, so continuing in \cref{eq:prfWrel:10},
    \begin{align*}
        \int_{B(0,\varepsilon)} g_\varepsilon(y) &\left( \int_{B(y,R)} (\Delta W)_+(x) \diff x \right) \diff y\\
                           &= \int_{\R^d} g_\varepsilon(y) \left( \int_{B(0,R)} (\Delta W)_+(x) \diff x + O(\varepsilon) \right) \diff y \\
                           &\longrightarrow \int_{B(0,R)} (\Delta W)_+(x) \diff x,
    \end{align*}
    as $\varepsilon \to 0$. We can proceed similarly with $(\Delta W)_-$ and combine the results to obtain \[
        \int_{B(0,R)} (\Delta W \star g_\varepsilon)(x)\diff x \longrightarrow  \int_{B(0,R)} \Delta W(x) \diff x
        \] as $\varepsilon \to 0$. We have now shown that letting $\varepsilon \to 0$ in \cref{eq:prfWrel:11} yields \[
        -\eta_W' + \int_{B(0,R)} \Delta W(x) \diff x = c_d R ^{d-1} \left<\nabla W\right>(R) = -\frac{\left<\nabla W\right>(R)}{\left<\nabla W_N \right> (R)},
    \] where we recalled $\left<\nabla W_N \right> (R) = - c_d^{-1} R^{1-d}$. If we let $R \to 0$, then the LHS tends to $-\eta_W'$ because $\Delta W$ is locally integrable. Thus the limit on the RHS exists, so $\eta_W$ in \cref{eq:alphaWdef} is well-defined and equals $\eta_W'$. If we let $R \to \infty$, then the LHS tends to $-\eta_W + \int \Delta W$ by \cref{ass}(iii), which implies that $\alpha_W$ in \cref{eq:alphaWdef} is well-defined and that $\alpha_W  = \eta_W - \int (\Delta W)$ holds.
\end{proof}

\subsubsection{Proof of \cref{lem:divbound}}

We already proved \cref{eq:divbound} just before stating \cref{lem:divbound}, so it remains to prove sharpness.
\begin{proof}[Proof of \cref{lem:divbound}, Sharpness]
     Let $c > c_W$, $x_0\in\R^d$ and $M > 0$. We will construct a function $g\in C^\infty_{c+}(\R^d)$ such that $g$ has a global maximum of height $M$ at $x_0$ and $\nabla \cdot (\nabla W \star g)(x_0) > -cg(x_0)$. Because this inequality is linear in $g$, we may assume $M = 1$. Put $U\coloneqq \left\{x\in \R^d\setminus \left\{ 0 \right\} \colon  \Delta W(x) > 0 \right\}$, which is an open, possibly empty subset of $\R^d \setminus \left\{ 0 \right\} $. For $\varepsilon > 0$ let \[
        U_\varepsilon \coloneqq \left\{ x\in U\colon d(x,U^{c}) > \varepsilon \right\} ,
    \] which is also open and $\bigcup_{\varepsilon > 0} U_\varepsilon = U$ (note that we shrink, not grow $U$ by $\varepsilon$). Put $h_\varepsilon \coloneqq \varphi_{\varepsilon / 2} \star \ind_{U_\varepsilon \cap B(0,\varepsilon^{-1})}$ where $\varphi \in C_{c+}^\infty(B(0,1))$ with $\int \varphi = 1$ and $\varphi_\varepsilon(x) \coloneqq \varepsilon^{-d} \varphi(\varepsilon^{-1}x)$. Then $h_\varepsilon \in C_{c+}^\infty(\R^d)$ and
    \begin{equation}\label{eq:hepsmonotone}
        \ind_{U_{2\varepsilon} \cap B(0,(2\varepsilon)^{-1})} \le h_\varepsilon \le \ind_{U_{\varepsilon / 2}}.
    \end{equation} In particular $0 \le h_\varepsilon \uparrow \ind_{U}$ as $\varepsilon \to 0$, and, since $h_\varepsilon$ is supported on $U$ for all $\varepsilon > 0$, and by monotone convergence,
    \begin{align}\label{eq:heps}
        \int_{\R^d}\Delta W(x) h_\varepsilon(x) \dd x = \int_{\R^d} (\Delta W)_+(x) h_\varepsilon(x) \dd x \uparrow \int(\Delta W)_+,\quad \varepsilon \to 0.
    \end{align}
    Furthermore, by \cref{eq:hepsmonotone}, $h_\varepsilon \equiv 0$ on $B(0,\varepsilon / 2)$, in particular $h_\varepsilon(0) = \Delta h_\varepsilon(0) = 0$.

    Now let $f_\varepsilon \in C_{c+}^\infty(B(x_0,\varepsilon / 2))$ with $f_\varepsilon(x_0) = 1$, $\Delta f_\varepsilon(x_0) = 0$, $0 \le f_\varepsilon \le 1$. In particular $f_\varepsilon \to 0$ a.e.\ as $\varepsilon \to 0$, and by dominated convergence (recall that $\Delta W$ is locally integrable by \cref{ass}(ii)),
        \begin{align}\label{eq:feps}
            (\Delta W \star f_\varepsilon)(x) \to 0,\quad \varepsilon \to 0,
        \end{align}
            for any $x\in \R^d$. Now put \[
                g_\varepsilon \coloneqq h_\varepsilon(x_0 - \cdot ) + f_\varepsilon \in C_{c+}^\infty(\R^d),\qquad \varepsilon > 0.
        \] Recall that $h_\varepsilon(x_0-\cdot )$ is supported in $B(x_0, \varepsilon / 2)^{c}$ and $f_\varepsilon$ is supported in $B(x_0,\varepsilon / 2)$, and both are upper bounded by $1$, so $0 \le g_\varepsilon \le 1$. Furthermore, \[
        g_\varepsilon(x_0) = h_\varepsilon(0) + f_\varepsilon(x_0) = 1,\qquad \Delta g_\varepsilon(x_0) = \Delta h_\varepsilon(0) + \Delta f_\varepsilon(x_0) = 0,
        \] so $g_\varepsilon$ attains a global maximum at $x_0$, and by \cref{eq:heps,eq:feps},
        \begin{align*}
            (\Delta W \star g_\varepsilon)(x_0) = (\Delta W \star h_\varepsilon)(0) + (\Delta W \star f_\varepsilon)(x_0) \to \int (\Delta W)_+,\quad \varepsilon \to 0.
        \end{align*}
        Now choose $\varepsilon > 0$ so small that $(\Delta W \star g_\varepsilon)(x_0) > \int (\Delta W)_+ - (c-c_W)$, then
    \begin{align*}
        \nabla \cdot (\nabla W\star g_\varepsilon)
        = (\Delta W \star g_\varepsilon)(x_0) - \eta_W g_\varepsilon(x_0)
        &= (\Delta W \star g_\varepsilon)(x_0) - \int(\Delta W)_+ - c_W\\
        &> -c\\
        &= -cg_\varepsilon(x_0),
    \end{align*}
    where we used $\eta_W = c_W + \int (\Delta W)_+$, see \cref{eq:cW}.
\end{proof}

\subsubsection{Proof of \cref{lem:partial+,lem:rhoreg}}

\begin{proof}[Proof of \cref{lem:partial+}]
    It suffices to show that for any fixed $t_0 \in [0,T)$, and every $\varepsilon,\varepsilon' > 0$,
    \begin{equation}\label{eq:prfinf_1}
        \left\|g_t\right\|_\infty \le \left\|g_{t_0}\right\|_\infty + \int_{t_0}^{t} \Big(C(s) + \varepsilon\Big) \dd s + \varepsilon',\quad t\in [t_0,T].
    \end{equation} Indeed, taking $\varepsilon' \downarrow 0$, this implies that \cref{eq:prfinf_1} holds for all $t \in [t_0,T]$, $\varepsilon > 0$, and $\varepsilon' = 0$. Thus, for every $\varepsilon > 0$,
    \begin{equation*}
        \partial_t^+ \left\|g_t\right\|_\infty \Big\vert_{t = t_0} = \varlimsup_{t\downarrow t_0}\frac{\left\|g_t\right\|_\infty - \left\|g_{t_0}\right\|_\infty}{t-t_0} \le \varlimsup_{t\downarrow t_0}\frac{1}{t-t_0} \int_{t_0}^t \Big( C(s) + \varepsilon\Big) \dd s = C(t_0) + \varepsilon,
    \end{equation*} which implies the claim at $t_0$ by taking $\varepsilon \downarrow 0$.

    Now fix $t_0\in [0,T)$, $\varepsilon,\varepsilon' > 0$, and we show \cref{eq:prfinf_1}. We may assume that $t_0 = 0$. We want to show that 
    \begin{equation}\label{eq:prflem+:1}
        F(t) \coloneqq \left\|g_t\right\|_\infty - \left\|g_0\right\|_\infty - \int_0^t \Big(C(s) -\varepsilon\Big)\dd s - \varepsilon' \le 0
    \end{equation} for all $t \in [0,T]$. Define \[
    t_1 \coloneqq \sup \left\{ t \in [0,T]\colon F(t) \le 0 \right\}.
\] $F$ is continuous because $g \in C([0,T],L^\infty)$, and $F(0) = -\varepsilon' < 0$, so $t_1 > 0$. Now if $F(t) \le 0$ does not hold for all $t\in [0,T]$, then $t_1 < T$, so there would be $T > t_n \downarrow t_1 > 0$ such that $F(t) \le 0$ for all  $0 \le t \le t_1$ and $F(t_n) > 0$ for all $n\in \N$. Since $g_{t_n}$ is Lipschitz continuous and integrable it must vanish at infinity, so it attains its maximum at some $x_n \in \R^d$ and we have
    \begin{equation}\label{eq:prflem+:2}
    F(t_n) = g(t_n,x_n) - \left\|g_0\right\|_\infty - \int_0^{t_n} \Big(C(s) + \varepsilon\Big)\dd s - \varepsilon' > 0,\quad n\in \N.
\end{equation} In particular, $g(t_n,x_n) \ge \varepsilon'$ for all $n\in \N$, so $g(t_1,x_n) \ge \varepsilon' / 2$ for all $n\ge n_0$ for some $n_0\in \N$, because $g\in C([0,T],L^\infty)$. Then by Lipschitz continuity of $g(t_1,\cdot )$ in space, $\inf_{n\in \N} \int_{B(x_n,1)} g(t_1,y) \diff y > 0$, which would contradict integrability of $g(t_1,\cdot )$ if $(x_n)$ were unbounded. Hence $(x_n)$ must be bounded, without loss of generality already convergent to some $x_1\in \R^d$. Then, using \cref{eq:prflem+:1,eq:prflem+:2},
    \begin{align*}
        F(t_1) &= \left\|g_{t_1}\right\|_\infty - \left\|g_0\right\|_\infty - \int_0 ^{t_1} \Big(C(s) + \varepsilon\Big) \dd s - \varepsilon',\\
               &\text{ and }\\
        F(t_1) &= \lim_{n\to \infty} F(t_n) = g(t_1,x_1) - \left\|g_0\right\|_\infty - \int_0^{t_1} \Big( C(s) + \varepsilon\Big) \dd s - \varepsilon',\\
    \end{align*} so $\left\|g_{t_1}\right\|_\infty = g(t_1,x_1)$. Thus by assumption $\partial_t g(t_1,x_1) \le C(t_1)$, so
    \begin{align*}
        \forall x\in \R^d\colon g(t_1,x) - \left\|g_0\right\|_\infty - \int_0^{t_1} \Big(C(s) + \varepsilon\Big)\dd s - \varepsilon' \le F(t_1) &= 0,\\
        \partial_t\Big[ g(\cdot ,x_1) - \left\|g_0\right\|_\infty - \int_0^t \Big(C(s)+\varepsilon\Big)\dd s - \varepsilon'\Big]_{t=t_1} \le C(t_1) - \Big(C(t_1) + \varepsilon\Big) = -\varepsilon &< 0.
    \end{align*}
    But this implies, since $\partial_t g$ is jointly continuous, that the second inequality holds in $[t_1,t_1+\delta) \times B(x_1,\delta)$ for some $\delta > 0$, so in fact $g(t,x) < \left\|g_0\right\|_\infty + \int_0^t \Big(C(s)+\varepsilon\Big) \dd s + \varepsilon'$ for all $(t,x) \in(t_1,t_1+\delta) \times B(x_1,\delta)$, which contradicts the fact \cref{eq:prflem+:2} that $g(t_n,x_n) > \left\|g_0\right\|_\infty + \int_0^t \Big( C(s) + \varepsilon\Big) \dd s + \varepsilon'$ for all $n\in \N$.
\end{proof}

\begin{proof}[Proof of \cref{lem:rhoreg}]
    Let $\rho_0\in \mathcal{X}_+$, $T \in (0,T^\star)$, and fix $\gamma \in (2,2+\gamma_f)$. Then by \cref{thm:regularity}, $\rho\in C([0,T],\mathcal{X}_+) \cap C((0,T],\mathcal{X}^\gamma)$, in particular $\rho_t$ is Lipschitz continuous for $t \in (0,T]$. It remains to prove that $\partial_t \rho$ is jointly continuous on $[\delta,T]\times \R^d$ for any $\delta > 0$. By the instant regularisation proved in \cref{thm:regularity}, $\rho_\delta \in \mathcal{X}^\gamma$, so we may as well assume that $\rho_0\in \mathcal{X}^\gamma$ and $\delta = 0$. Then we know that $\rho$ solves \cref{eq:classicalpde} on $[0,1]$, so 
    \begin{equation}\label{eq:prfrhoreg:1}
        \partial_t \rho_t = \Delta \rho_t + \nabla \cdot (\rho_t (\nabla W \star \rho_t)) + f.
    \end{equation} Using $\rho \in C([0,1],\mathcal{X}^{\gamma})$ and \cref{lem:W1W2,eq:LWconv,lem:etaW}, it is straightforward to confirm through direct calculation that the RHS in \cref{eq:prfrhoreg:1} and therefore $\partial_t \rho \in C([0,1],\mathcal{W}^{\gamma-2,\infty})$, where $\mathcal{W}^{\gamma-2,\infty} = C^{0,\gamma-2}(\R^d)$ (see \cref{eq:Wsinfty} in \cref{app:sobolev}) with $\gamma-2 \in (0,1)$. In particular, $\partial_t \rho$ is continuous in time, and $(\gamma-2)$-Hölder continuous in space uniformly on $[0,T]$. This implies joint continuity
\end{proof}

\subsubsection{Asymptotics of solution in absence of repulsion}

We close by presenting a proof of a simple claim made in the introduction regarding the asymptotics of \cref{eq:classicalpde} with $W \equiv 0$.
\begin{lemma}\label{lem:whenWis0}
    Let $f\in \mathcal{X}_+$ not a.e.\ zero, and put $\rho_t \coloneqq \int_0^t G_s \star f\dd s$ for $t \ge 0$.
    \begin{enumerate}
        \item If $d\le 2$, then $\rho_t\uparrow \infty$ locally uniformly.
        \item If $d\ge 3$, then $\rho_t \uparrow \rho$ where $\rho \in L^\infty$.
    \end{enumerate}
\end{lemma}
\begin{proof}
    \begin{enumerate}
        \item We may assume $\int_{B(0,1)} f(x) \dd x \ge 1$. Then, for any $x\in \R^d$,
            \begin{align*}
                \rho_t(x) = \int_0^t G_s \star f(x) \dd s
                &\ge \int_0^t \int_{B(0,1)} f(y) G_s(x-y) \dd y \dd s\\
                &\ge \int_0^t (2\pi s)^{-d / 2}\e^{-(\left| x \right| +1)^2 / (2s)} \dd s\\
                &\ge  \e^{-(\left| x \right| +1)^2 / 2} \int_1^t (2\pi s)^{-d / 2} \dd s,
            \end{align*}
            which if $d  \le 2$ goes to $\infty$ locally uniformly in $x$ as $t \to \infty$.
        \item We have $\rho_t \uparrow \rho \coloneqq \int_0^\infty G_s \star f \dd s = G \star f$ as $t \to \infty$, where \[
                G(x) \coloneqq \int_0^\infty G_s(x) \dd s = c(d) \left| x \right| ^{2-d},\quad x\in \R^d,
            \] for some $c(d) > 0$ is the well-known Green's function of the Laplace equation in $d \ge 3$. Then $G$ is integrable at the origin and bounded away from the origin, and $f$ is both bounded and integrable, so $G\star f$ is bounded. 
    \end{enumerate}
\end{proof}

\section{Conclusion and Outlook}\label{sec:outlook}
We established sharp conditions on the repulsive potential for a form of the maximum principle \cref{eq:cheapMP} and a strong notion of global boundedness \cref{eq:cheapMprop} to hold. The latter is especially interesting in light of the motivation from population biology -- see the introduction and \cref{app:superprocesses} -- because it gives a sufficient condition on $W$ for global boundedness of solutions in the sense that
\begin{equation}\label{eq:cheapbounded}
    \forall \rho_0\in \mathcal{X}_+\colon \sup_{t\ge 0}\left\|\rho_t\right\|_\infty < \infty.
\end{equation}
Note however that \cref{eq:cheapMprop} is a much stronger property than \cref{eq:cheapbounded}, which we would expect to hold under weaker assumptions: $c_W > 0$ necessitates both $\eta_W > 0$ and $\alpha_W > 0$, that is a singular repulsion and $|\nabla W(x)| \gtrsim |x|^{1-d}$ for large $|x|$. However, we would expect that \cref{eq:cheapbounded} should also hold for a sufficiently strong smooth repulsion, and it seems unlikely that $|\nabla W(x)| \gtrsim |x|^{1-d}$ is sharp; for example if $d \ge 3$, then \cref{eq:cheapbounded} already holds with $W = 0$. Ongoing work tentatively suggests that the critical strength of the repulsion for \cref{eq:cheapbounded} is $|\nabla W(x)| \gtrsim |x|^{-1}$ if $d = 1$ and $|\nabla W(x)| \gtrsim |x|^{-3}$ if $d = 2$ (and none if $d \ge 3$).

Another direction for future research would be to study non-negative steady states associated with \cref{eq:classicalpde}, whose existence we would generally expect to be related to \cref{eq:cheapbounded}.


\section*{Acknowledgements}
I would like to thank my supervisor Alison Etheridge for many lively discussions, guidance, encouragement, and helpful feedback. I would like to thank Jos\'e Carrillo for many insightful conversations, valuable feedback, for pointing out useful literature, and for advice regarding publication of the project. I thank the referees, particularly for pointing out important missing references, and for helping to improve the presentation of the paper.

\appendix
\section*{Appendix}
\section{Connection with Stability of Population Dynamics}\label{app:superprocesses}

We present here in some more detail the connection between the asymptotic behaviour of solutions to \cref{eq:classicalpde} and long-term (in-)stability of branching particle systems (BPS). We will explain how the dichotomy observed in the introduction is related to the fact that an ordinary BPS (without immigration or interaction) started from infinite mass is unstable in dimensions $d \le 2$ in the sense that the process' mass concentrates, as time goes on, in increasingly large ``clumps'', with space in between growing increasingly empty.

\subsection*{SuperBrownian Motion}

Long-term instabilities of a BPS, for now without immigration or interaction, are due to random fluctuations in the branching mechanism, so if we want to study them using a scaling limit, then the scaling needs to retain stochasticity. Indeed, the hydrodynamic rescaling just leads to the heat equation, which has stable long-term behaviour in any dimension. A well-studied approach to retain stochasticity is to scale up the branching rate at the same time as the particle density, leading to a measure-valued process called \emph{superBrownian motion (SBM)} \cite{alisonsuperprocesses,perkinssuperprocesses}.
Formally, it solves the stochastic partial differential equation (SPDE)
\begin{equation}\label{eq:intro:spde}
\dd X_t = \frac{1}{2}\Delta X_t \dd t + \sqrt{\gamma X_t} \dd \mathcal{W}_t,
\end{equation}
for a space-time white noise $\mathcal{W}$, and a parameter $\gamma > 0$ called the \emph{branching variance}; note that setting $\gamma = 0$ recovers the hydrodynamic limit (the heat equation). If $d \ge 2$, then $X_t$ is singular w.r.t.\ Lebesgue measure, and \cref{eq:intro:spde} is ill-posed and has to be replaced with a martingale problem.

\subsection*{The Pain in the Torus}

We will now explain heuristically why a SBM started from infinite mass, or from finite mass but conditioned on survival, is unstable in dimensions $d \le 2$, and how this is related to \cref{eq:classicalpde}. In short, for large times $t$, an increasingly small number of individuals that were alive at time zero will be ancestral to the entire population at time $t$, and if $d \le 2$ then the diffusion is not ``fast enough'' to disperse and spread these large families, and they form well-separated clumps. It is this that underpins the problem famously dubbed ``the pain in the torus'' by Felsenstein \cite{felsenstein}; see also Kallenberg \cite{kallenbergtree} (esp.\ Cor.\ 6.5) for similar observations in the context of cluster fields.


Let us now consider an SBM started from Lebesgue measure, and make this idea a bit more precise. We cut $\R^d$ into a grid of unit sized cubes, and regard the initial mass in each of them as one family. Due to the independent branching, we can let each of the families evolve independently from each other, and obtain the process started from Lebesgue measure as their superposition (this is called the branching property, see e.g.\ \cite[p.\ 2]{alisonsuperprocesses}). Each family is a critical branching process started from finite mass, so the probability that it is still alive at time $n$ is proportional to $1 / n$, and, if alive, its expected size is proportional to $n$ \cite[Thm.\ II.1.1]{perkinssuperprocesses}. Hence, in expectation, after $n$ units of time all but every $n$'th family has gone extinct, and each of the living families consists of order $n$ individuals. Due to their diffusive movement, each family will have spread over an area of radius $\sim\sqrt{n} $, hence the population density of any surviving family is $\sim n^{1 - d / 2}$. In the critical case $d =2$, this crude heuristic misses a factor $\log n$ (c.f.\ \cref{eq:appmassaccumulation} below), so the density of the surviving families diverges as $n \to \infty$ if and only if $d \le 2$, in which case they form separated clumps.

This means that the dichotomy between stable long-term dynamics and clumping can really be understood as the dichotomy between unbounded and bounded asymptotic population density of a single surviving family, that is, a finite mass SBM conditioned on survival. It is a classical result due to Evans \cite{evansimmortal} that the distribution of this process is that of a single ``immortal particle'' that follows the path of a Brownian motion and throws off mass at a constant rate, which then evolves like an ordinary SBM, independent of the immortal particle. This is not unexpected: In the unconditioned process, at large times the entire population will have descended from increasingly few ancestors that were alive at time zero, until eventually none remain and the process goes extinct; the conditioning imposes that one of those ancestors---the immortal particle---will never perish. If $(Z_t)$ is a Brownian motion that denotes the path of the immortal particle, then the SBM $(X_t)$ conditioned on survival, which we may now as well start from zero, formally satisfies the SPDE \[
    \dd X_t = \left( \frac{1}{2}\Delta X_t + \gamma \delta_{Z_t}\right)\diff t + \sqrt{\gamma X_t} \dd \mathcal{W}_t,\quad X_0 = 0,
    \] where $\delta_z$ is the Dirac delta at $z\in \R^d$. Then the mean measure conditional on $(Z_t)$ has a density $\rho_t$ that solves $\partial_t \rho = \frac{1}{2}\Delta \rho +\gamma \delta_{Z_t}$, so $\rho_t = \gamma \int_0^t G_{t-s}(\cdot -Z_s)\diff s$. If we centre the process around the immortal particle by putting \[
    \widetilde{\rho}_t(\cdot ) \coloneqq \rho_t(\cdot +Z_t) =\gamma \int_0^t G_{t-s}(\cdot +Z_t-Z_s)\diff s,
    \] which has expectation $\gamma \int_0^t G_{2s} (\cdot ) \diff s$, then we can find 
    \begin{equation}\label{eq:appmassaccumulation}
        \mathbb{E} \left[ X_t(B(Z_t,1)) \right] = \smashoperator{\int\limits_{B(0,1)}} \mathbb{E} \left[ \widetilde{\rho}_t(x) \right] \diff x =\gamma \int_0^t \smashoperator[r]{\int\limits_{B(0,1)}} G_{2(t-s)}(x) \diff x\diff s \propto
    \begin{cases}
        \sqrt{t} , & d = 1,\\
        \log t , & d = 2,\\
        1,  & d \ge 3.
    \end{cases}
    \end{equation}
    Therefore, at least in expectation, mass accumulates in the vicinity of the immortal particle for large times if $d \le 2$, and remains bounded if $d \ge 3$. This recovers the picture painted in the beginning of the section: In dimensions $d \le 2$, a SBM started from infinite mass concentrates in increasingly few large clumps (centred around the ancestors of the surviving families), with space in between growing increasingly empty.

\subsection*{Introducing Repulsion}
As a model for a spatially evolving population (for which $d = 2$ is the most natural dimension), this is very unrealistic, and a better model should reflect stable long-term population dynamics. 
One of the most obvious reasons this clumping phenomenon does not occur in nature is that real individuals do not behave independently from surrounding individuals, as is assumed in the model underlying SBM. Indeed, a high population density leads to resource scarcity and decreases the average number of offspring, and causes migration away from the overcrowded area. The former effect has already been successfully integrated into the SBM model and been shown to lead to stable long-term dynamics \cite{localregulation}. The latter however, has not yet been studied in this context. A natural way to implement this is to introduce a pairwise repulsion between individuals, which corresponds to the term involving $W$ in \cref{eq:classicalpde}.

If we again consider a single surviving family, that is a finite mass superprocess conditioned on survival, then again this will be described by an immortal particle that constantly immigrates mass into the system, which will now be repulsed from the mass it throws off. Formally, we arrive at
\begin{align*}
    \dd X_t &= \left( \frac{1}{2}\Delta X_t + \nabla \cdot (X_t \nabla W \star X_t) +\gamma \delta_{Z_t} \right)\diff t+ \sqrt{\gamma X_t} \dd \mathcal{W}_t,\\
    \diff Z_t &= \diff B_t - \nabla W \star X_t(Z_t) \diff t,
\end{align*}
for independent space-time white noise $\mathcal{W}$ and Brownian motion $B$. This turns out to be a very complicated process, and a natural first step is to study it without the noise; if the equation were linear, this would be the same as taking expectations conditional on $(Z_t)$. If we also replace the Dirac immigration with a bounded function centred on $Z_t$---which should not change the behaviour of the system with regards to clumping behaviour, but makes the equation more regular---then we arrive exactly at \cref{eq:classicalpde} with a time-dependent immigration (recall \cref{rem:timedependentf}), and the question we want to answer is under what assumptions on the repulsion does its solution exhibit bounded long-term behaviour in dimensions one and two.

\section{Fractional Sobolev Spaces}\label{app:sobolev}
We give a minimal definition of fractional Sobolev spaces, and refer the reader to \cite{hitchhiker,sobolev2,sobolev3} for detailed introductions. For $p\in [1,\infty)$, $s\in (0,1)$, and $u\colon \R^d\to \R$ measurable let 
\begin{equation}\label{eq:Wspbrac}
    [u]_{\mathcal{W}^{s,p}} \coloneqq \left( s(1-s) \int_{\R^d}\int_{\R^d} \frac{\left| u(x)-u(y) \right| ^p}{\left| x-y \right| ^{d+sp}}\diff x\diff y \right) ^{1 / p},
\end{equation} and
\begin{equation}\label{eq:Wsinfty}
    [u]_{\mathcal{W}^{s,\infty}} = \sup_{x\neq y} \frac{\left| u(x)-u(y) \right| }{|x-y|^s}.
    \end{equation} For $k\in \N$, $s\in (k,k+1)$, and $p\in [1,\infty]$, let \[
    [u]_{\mathcal{W}^{s,p}} \coloneqq \sum_{|\alpha| = k} [\partial^\alpha u]_{\mathcal{W}^{s-k,p}},
\] with the usual notation for multi-indices $\alpha$. Then
\begin{equation}\label{eq:Wsp}
    \left\|u\right\|_{\mathcal{W}^{s,p}} = \left( \left\|u\right\|_{\mathcal{W}^{\left\lfloor s \right\rfloor ,p}}^p + [u]_{\mathcal{W}^{s,p}}^p  \right) ^{1 / p}
\end{equation} for $p \in [1,\infty)$, and 
\begin{equation}\label{eq:Wsy}
    \left\|u\right\|_{\mathcal{W}^{s,\infty}} = \left\|u\right\|_{\mathcal{W}^{\left\lfloor s \right\rfloor ,\infty}} + [u]_{\mathcal{W}^{s,\infty}},
\end{equation} define the fractional Sobolev norms.

For two normed spaces write $A \hookrightarrow B$ if $A\subset B$ with continuous inclusion.
\begin{lemma}\label{lem:appalphamonotone}
    If $0\le \gamma \le \beta$ and $p\in [1,\infty]$, then $\mathcal{W}^{\beta,p} \hookrightarrow \mathcal{W}^{\gamma,p}$.
\end{lemma}
\begin{proof}
    Assume $0 < \gamma < \beta$, otherwise there is nothing to show. Then the claim follows because $\mathcal{W}^{\gamma,p}$ can be written as interpolation space between $L^p$ and $\mathcal{W}^{\beta,p}$, so $\left\|\cdot \right\|_{\mathcal{W}^{\gamma,p}} \le C ( \left\|\cdot \right\|_{L^p} + \left\|\cdot \right\|_{\mathcal{W}^{\beta,p}} ) \le C \left\|\cdot \right\|_{\mathcal{W}^{\beta,p}}$. See \cite[Appendix A]{joseheatkernel} for details on interpolation spaces in the context of Sobolev norms.
\end{proof}

\begin{lemma}
    \hangindent\leftmargini 
    \label{lem:embedding}
    \textup{(i)} If $\gamma > 0$, $\gamma \not\in \N$, then $\mathcal{X}^{\gamma}\hookrightarrow \mathcal{W}^{\gamma,\infty} = C^{\left\lfloor \gamma \right\rfloor ,\gamma-\left\lfloor \gamma \right\rfloor }(\R^d)$.
    \begin{enumerate}
        \setcounter{enumi}{1}
        \item If $\gamma \ge 1$, then $\mathcal{X}^{\gamma} \hookrightarrow \mathcal{W}^{\gamma,\infty} \hookrightarrow C^{\left\lfloor \gamma \right\rfloor -1,1}(\R^d)$.
    \end{enumerate}
\end{lemma}
\begin{proof}
    First note that $\mathcal{X}^{\gamma} \hookrightarrow \mathcal{W}^{\gamma,\infty}$ by definition of $\mathcal{X}^\gamma$, see \cref{eq:defWgamma}.
    \begin{enumerate}
        \item For any $k\in \N_0$ and $\gamma \in (k,k+1)$, if $f\colon \R^d \to \R$ is $k$-times differentiable then $\left\|f\right\|_{\mathcal{W}^{\gamma,\infty}} = \left\|f\right\|_{C^{k,\gamma-k}(\R^d)}$ by definition, see \cref{eq:Wsinfty,eq:Wsy,eq:Ckbeta}. Hence we only have to prove that any $f\in \mathcal{W}^{\gamma,\infty}$ for $\gamma > 0$ is $\left\lfloor \gamma \right\rfloor $-times differentiable.

            There is nothing to prove for $\gamma \in (0,1)$, so suppose the claim is true for $\gamma \in (0,k)\setminus \N$ for some $k\in \N_0$ and let $f\in \mathcal{W}^{\gamma,\infty}$ for some $\gamma \in (k,k+1)$. Put $\beta \coloneqq \gamma - \left\lfloor \gamma \right\rfloor $. Then $f$ is continuous because $\mathcal{W}^{\gamma,\infty}\subset \mathcal{W}^{1,\infty} = C^{0,1}(\R^d)$ (the last equality is a well-known theorem). Furthermore $f \in \mathcal{W}^{\gamma,\infty} \subset \mathcal{W}^{1+\beta,\infty}$, so $f$ has partial weak derivatives of first order that are $\beta$-Hölder continuous, so they are in fact proper derivatives
            \footnote{Let $d = 1$ and $\varphi_\varepsilon$ an approximation to unity in $\R$, $x > 0$, and put $\psi_\varepsilon \coloneqq \varphi_\varepsilon \star \ind_{[0,x]}$, then \[
                    \int_0^x f'(y)\diff y \leftarrow \int_{\R} \psi_\varepsilon(y) f'(y) \diff y = - \int_{\R} \psi_\varepsilon'(y) f(y) \diff y = f(x) - f(0),
            \] so $f(x) = f(0) + \int_0^x f'(y) \diff y$. Analogously for $x < 0$. Since $f'$ is continuous, $f$ is in fact differentiable in the classical sense with derivative $f'$. This also works in $d \ge 2$.}.
            By the induction hypothesis, all first partial derivatives are themselves $(\left\lfloor \gamma \right\rfloor -1)$ times differentiable.
        \item Since $\mathcal{W}^{\gamma,\infty} \hookrightarrow \mathcal{W}^{\left\lfloor \gamma \right\rfloor ,\infty}$, it suffices to show that $\mathcal{W}^{k,\infty} \hookrightarrow C^{k-1,1}(\R^d)$ for all $k\in \N$. This is a well-known theorem for $k = 1$. Suppose now it is proved for some $k\in \N$, and let $f\in \mathcal{W}^{k+1,\infty}$. Then by (i), $f$ has proper first partial derivatives, and they are in $\mathcal{W}^{k,\infty} \hookrightarrow  C^{k-1,1}(\R^d)$, so in fact $f$ is $k$-times differentiable and, by the induction hypothesis,
            \begin{align*}
        \left\|f\right\|_{\mathcal{W}^{k+1,\infty}} 
        &= \left\|f\right\|_\infty + \sum_{i=1}^d \left\|\partial_{i} f\right\|_{\mathcal{W}^{k,\infty}} \\
        &\le \left\|f\right\|_\infty + C\sum_{i=1}^d \left\|\partial_i f\right\|_{C^{k-1,1}} \\
        &\le C \left\|f\right\|_{C^{k,1}}.
    \end{align*}
    \end{enumerate}
\end{proof}

\begin{lemma}\label{lem:appalphabeta}
    Let $0\le \gamma < \beta$. Then there is $C > 0$ such that the following hold.
    \begin{enumerate}
        \item[(i)] If $f,g\in \mathcal{W}^{\gamma,\infty}$, then $fg \in \mathcal{W}^{\gamma,\infty}$ and $\left\|fg\right\|_{\mathcal{W}^{\gamma,\infty}} \le C \left\|f\right\|_{\mathcal{W}^{\gamma,\infty}} \left\|g\right\|_{\mathcal{W}^{\gamma,\infty}}$,
        \item[(ii)] If $f\in \mathcal{W}^{\gamma,1}, g\in \mathcal{W}^{\beta,\infty}$, then $fg\in \mathcal{W}^{\gamma,1}$ and $\left\|fg\right\|_{\mathcal{W}^{\gamma,1}}\le C \left\|f\right\|_{\mathcal{W}^{\gamma,1}} \left\|g\right\|_{\mathcal{W}^{\beta,\infty}}$.
    \end{enumerate}
\end{lemma}
\begin{proof}
    If $\gamma\in \N_0$ and $p\in \left\{1,\infty \right\} $, then for any multi-index $\left| \alpha \right| \le \gamma$,
    \begin{align*}
        \left\|\partial^\alpha (fg)\right\|_{L^p}
        &= \left\|\sum_{\alpha_1 + \alpha_2 = \alpha} (\partial^{\alpha_1} f)(\partial^{\alpha_2}g)\right\|_{L^p} \le \sum_{\alpha_1 + \alpha_2 = \alpha} \left\|\partial^{\alpha_1} f\right\|_{L^p} \left\|\partial^{\alpha_2}g\right\|_{L^\infty}\\
        &\le C \left\|f\right\|_{\mathcal{W}^{\gamma,p}} \left\|g\right\|_{\mathcal{W}^{\gamma,\infty}}.
    \end{align*}
    and both claims follow. If $\gamma \not\in \N_0$, we need to show the claim with the LHS replaced by $\left[ fg \right] _{\mathcal{W}^{\gamma,p}}$. By a similar application of the product rule it suffices to consider $\gamma \in (0,1)$. If $p = \infty$, then
    \begin{align*}
        \left[ fg \right] _{\mathcal{W}^{\gamma,\infty}}
        &= \sup_{x\neq y} \frac{\left| f(x)g(x) - f(y)g(y) \right| }{\left| x-y \right|^\gamma }\\
        &\le \left\|f\right\|_{L^\infty} \left[ g \right] _{\mathcal{W}^{\gamma,\infty}} + \left\|g\right\|_{L^\infty} \left[ f \right] _{\mathcal{W}^{\gamma,\infty}}\\
        &\le 2\left\|f\right\|_{\mathcal{W}^{\gamma,\infty}} \left\|g\right\|_{\mathcal{W}^{\gamma,\infty}}.
    \end{align*}
    If $p = 1$, then
    \begin{align*}
        \left[ fg \right] _{\mathcal{W}^{\gamma,1}}
        &= \gamma(1-\gamma) \int\int \frac{\left| f(x)g(x)-f(y)g(y) \right| }{\left| x-y \right| ^{d + \gamma}}\dd x\dd y \\
        &\le C \left\|g\right\|_{L^\infty} \left[ f \right] _{\mathcal{W}^{\gamma,1}} + \int \int \left| f(y) \right| \frac{\left| g(x) - g(y) \right| }{\left| x-y \right| ^{d+\gamma}}\dd x\dd y.
    \end{align*}
    We split the integral according to $\left| x-y \right| \le 1$ or $> 1$. The former contribution can be bounded by
    \begin{align*}
        \left[ g \right] _{\mathcal{W}^{\beta,\infty}} \iint\limits_{\left| x-y \right| \le 1} \frac{\left| f(y) \right| }{\left| x-y \right|^{d + \gamma - \beta}}\dd x\dd y\le C \left[ g \right] _{\mathcal{W}^{\gamma,\infty}} \left\|f\right\|_{L^1},
    \end{align*}
    where we used that $\int_{B(0,1)}\left| x \right| ^{-d+\beta-\gamma} \dd x \le C \int_0^1 r^{-1+\beta - \gamma} \dd r < \infty$. The contribution with $|x-y| > 1$ can be bounded by
    \begin{align*}
        2\left\|g\right\|_{L^\infty}\iint _{\left| x-y \right| > 1} \frac{\left| f(y) \right| }{\left| x-y \right| ^{d+\gamma}}\dd x\dd y \le C \left\|g\right\|_{L^\infty} \int \left| f(y) \right| \dd y = \left\|g\right\|_{L^\infty} \left\|f\right\|_{L^1},
    \end{align*}
    where we used that $\int_{\R^d\setminus B(0,1)} \left| x \right| ^{-\gamma-d} \dd x \le C  \int_1^\infty r^{-1-\gamma}\dd r < \infty$.
\end{proof}

\bibliographystyle{plain}
\bibliography{references}

\end{document}